\documentclass{amsart}
 
\usepackage{amssymb}
\usepackage[capitalise]{cleveref}
\usepackage{amsthm}
\usepackage{ mathrsfs }

\usepackage{tikz-cd}

\newtheorem{theorem}{Theorem}[section]
\newtheorem{proposition}[theorem]{Proposition}
\newtheorem{corollary}[theorem]{Corollary}
\newtheorem{lemma}[theorem]{Lemma}

\newtheorem{step}{Step}

  \theoremstyle{definition}
\newtheorem{definition}[theorem]{Definition}
\newtheorem{example}[theorem]{Example}
\newtheorem{remark}[theorem]{Remark}
\newtheorem{question}[theorem]{Question}

\DeclareMathOperator{\dist}{\mathsf{dist}}
\DeclareMathOperator{\Hdist}{\mathsf{hdist}}
\DeclareMathOperator{\Ends}{\mathsf{E}}

\newcommand{\R}{\mathbb{R}}    
  
\newcommand{\Z}{\mathbb{Z}}    
\newcommand{\N}{\mathbb{N}}

\newcommand{\calP}{\mathcal{P}}  
\newcommand{\calQ}{\mathcal{Q}}

 \newcommand\Gsetminus{- \kern-0.7em{{}^c}\ }

\begin{document}

\title[]{Quasi-isometric rigidity of subgroups and Filtered ends}

\author{Eduardo Mart\'inez-Pedroza}
\address{Memorial University  St. John's, Newfoundland and Labrador, Canada}
\email{eduardo.martinez@mun.ca}

\author{Luis Jorge S\'anchez Salda\~na}
\address{Departamento de Matem\' aticas, Facultad de Ciencias, Universidad Nacional Aut\'onoma de M\'exico, Circuito Exterior S/N, Cd. Universitaria, Colonia Copilco el Bajo, Delegaci\'on Coyoac\'an, 04510, M\'exico D.F., Mexico}
\email{luisjorge@ciencias.unam.mx}


\date{\today}



\begin{abstract}
Let $G$ and $H$ be quasi-isometric finitely generated groups and let $P\leq G$; is there a subgroup $Q$ (or a collection of subgroups) of $H$ whose left cosets coarsely reflect the geometry of the left cosets of $P$ in $G$? We explore sufficient conditions for a positive answer.  

The article consider pairs of the form $(G,\mathcal{P})$ where $G$ is a finitely generated group and $\mathcal{P}$ a finite collection of subgroups, there is a notion of quasi-isometry of pairs, and quasi-isometrically characteristic collection of subgroups. A subgroup is qi-characteristic if it belongs to a qi-characteristic collection. Distinct  classes of qi-characteristic collections of subgroups have been studied in the literature on quasi-isometric rigidity, we list in the article some of them and provide other examples. 

The first part of the article proves:   if $G$ and $H$ are finitely generated quasi-isometric groups  and $\mathcal{P}$ is a qi-characteristic collection of subgroups of $G$, then there is a collection of subgroups $\mathcal{Q}$ of $H$ such that $ (G, \mathcal{P})$ and $(H, \mathcal{Q})$ are  quasi-isometric pairs.

The second part of the article  studies the number of filtered ends $\tilde e (G, P)$ of a pair of groups, a notion introduced by Bowditch, and provides an  application of our main result: if $G$ and $H$ are quasi-isometric groups and $P\leq G$ is qi-characterstic, then there is $Q\leq H$ such that $\tilde e (G, P) = \tilde e (H, Q)$.

\end{abstract}
 
\maketitle

\section{Introduction}

Let $G$ be a finitely generated group with a chosen word metric $\dist_G$, and denote the Hausdorff distance between subsets of $G$ by $\Hdist_G$. If $\mathcal{P}$ is a finite collection of subgroups, then $G/\mathcal{P}$ denote the set of left cosets $gP$ with $g\in G$ and $P\in \mathcal{P}$. In this article, we consider pairs of the form $(G, \mathcal{P})$.

Consider two pairs $(G, \mathcal{P})$ and $(H, \mathcal{Q})$. An $(L,C)$-quasi-isometry $q\colon G \to H$ is an \emph{$(L,C,M)$-quasi-isometry of  pairs}  $q\colon (G, \mathcal{P})\to (H, \mathcal{Q})$ if the relation
\begin{equation*}\label{eq:def-qi-relation} \{ (A, B)\in G/\mathcal{P} \times H/\mathcal{Q} \colon \Hdist_H(q(A), B) <M  \} \end{equation*}
satisfies that the projections into $G/\mathcal{P}$ and $H/\mathcal{Q}$ are surjective.

A collection of subgroups $\mathcal{P}$ of $G$ is  \emph{quasi-isometrically characteristic} (or shorter  \emph{qi-characteristic}) in $G$ if $\mathcal{P}$ is finite, each $P\in \mathcal{P}$ has finite index in its commensurator, and 
every  $(L,C)$-quasi-isometry $q\colon G\to G$ is an $(L,C,M)$- quasi-isometry of pairs $q\colon (G,\mathcal{P})\to (G,\mathcal{P})$ where $M=M(G, \mathcal{P}, L,C)$, see~\cref{def:qicharacteristic} and~\cref{thm:qi-characteristic-groups2}.

In the study of quasi-isometric rigidity,   qi-characteristic collections appear in the literature. For example:
\begin{itemize}
    \item Kapovich and Leeb~\cite[Theorem 1.1 and Theorem 4.10]{KaLe97} proved that the geometric decomposition of a Haken manifold is preserved by quasi-isometries. The geometric components induce a qi-characteristic collection of the fundamentamental group of the manifold. 
    
    \item  Behrstock, Dru\c{t}u and Mosher~\cite[Theorems 4.1 and 4.8]{BDM09} proved that relative hyperbolicity with respect to non-relatively hyperbolic groups is a quasi-isometry invariant. Their results imply that the collection of maximal parabolic subgroups is qi-characteristic, see \cref{cor:BDM-RelHyp}.
 
    \item Lafont, Frigerio, and Sisto  \cite[Lemma~2.19 and Proposition~8.35]{LFS15} proved that in an irreducible higher graph manifolds the collection of walls is invariant under quasi-isometry, and the corresponding subgroups form a qi-characteristic collection.
 \end{itemize}

The main result of the first part of this article can be interpreted as an abstraction of a common technique used in the study of quasi-isometric rigidity, and in particular, in the works cited above. 

\begin{theorem}[\cref{cor:characteristic:groups}]\label{intro2:cor:characteristic:groups}
Let $G$ be a finitely generated group, let $\mathcal{P}$ be a finite qi-characteristic collection of subgroups of $G$.  If $H$ is a finitely generated group and $q\colon G \to H$ is a quasi-isometry, then  there is a qi-characteristic collection of subgroups $\mathcal{Q}$ of $H$ such that $q\colon (G, \mathcal{P}) \to (H, \mathcal{Q})$ is a quasi-isometry of pairs.
\end{theorem}

For additional examples of qi-characteristic collections see \cref{prop:example01}. The proof of  \cref{intro2:cor:characteristic:groups} is the content of~\cref{sec:qi-characteristic-Thm}, and can be described as follows.  The left multiplication action of $H$ on itself by isometries is pushed  via the quasi-isometry $q$ and its quasi-inverse to a quasi-action of $H$ on $G$ with uniform constants. The qi-characteristic hypothesis implies that there is a finite collection of left cosets $\mathcal{F}$ in $G/\mathcal{P}$ whose $H$-translates  reach all elements of $G/\mathcal{P}$ up to uniform Hausdorff distance. For each $gP\in \mathcal{F}$, the collection of elements of $H$ that fixed $gP$ up to finite Hausdorff distance turns out to be a subgroup. In this way the subgroups in $\mathcal{Q}$ arise as $H$-stabilizers of the elements in $\mathcal{F}$. Then an argument shows that $q\colon (G,\mathcal{P}) \to (H, \mathcal{Q})$ is a quasi-isometry of pairs. Since $\mathcal{P}$ is qi-characteristic in $G$, then it   follows that $\mathcal{Q}$ is qi-characteristic in $H$ as well. 

We traced back the idea of the argument proving \cref{intro2:cor:characteristic:groups} to the work of Schwartz~\cite{Sch95} related to quasi-isometric rigidity of finite volume hyperbolic manifolds. Our argument follows patterns in  the works of Kapovich and Leeb~\cite{KaLe97} on quasi-isometric rigidity of Haken manifolds,  Dru\c{t}u and Sapir~\cite[\S 5.2]{DS05} on quasi-isometric rigidity of relative hyperbolicity, and   Mosher, Sageev and White~\cite{MSW11} on quasi-isometric rigidity of fundamental groups of finite graphs of groups.  Recent work by A. Margolis~\cite{Ma19} contains an analogous statement to \cref{intro2:cor:characteristic:groups} in the case that $G$ is a Poincare duality group and $\calP$ consists of a single subgroup that is \emph{almost normal}, a condition, that is in a sense opposite to qi-characteristic.

As a sample application of ~\cref{intro2:cor:characteristic:groups}, we prove a quasi-isometric rigidity result for the number filtered ends over qi-subcharacteristic subgroups which are defined in the next paragraph. Let $G$ be a finitely generated group and let $P$ be a subgroup.  The  \emph{number of filtered ends $\tilde e(G,P)$} of the pair $(G, P)$ was introduced by Bowditch~\cite{Bow02}, under the name of \emph{coends}, in his study of JSJ splittings of one-ended groups. The number of filtered ends coincides with \emph{the algebraic  number of  ends of the pair $(G,P)$} introduced by  Kropholler and Roller~\cite{KrRo89}, see~\cite{Bow02} for the equivalence. The number of filtered ends does not coincide with the  \emph{number of relative ends $e(G,P)$} introduced by Houghton~\cite{Hou74}, but there are several relations including the inequality $e(G,P) \leq \tilde e(G,P)$ and equality in the case that $P$ is  normal and finitely generated, for an account see Geoghegan's book~\cite[Chapter~14]{Ge08}.

A subgroup $P$ of a finitely generated group $G$ is \emph{ qi-subcharacteristic} if $P$ belongs to a qi-characteristic collection of subgroups of $G$, see~\cref{thm:qi-characteristic-groups2}.

\begin{corollary}
\label{intro2:thm:final:application}
Let $G$ and $H$ be finitely generated quasi-isometric groups. If $P\leq G$ is a qi-subcharacteristic subgroup, then there is a qi-subcharacteristic subgroup $Q\leq H$ such that $\tilde e(G,P)= \tilde e(H,Q)$.
\end{corollary}

This corollary is a consequence of~\cref{intro2:cor:characteristic:groups} together with the main result of~\cref{section:ends:of:metric:pairs} that is stated below. 
We develop the notion of  filtered ends $\Ends (X, C)$ for pairs  where $X$ is a metric space and $C$ is a subspace, parallel to the treatment by Geoghegan~\cite{Ge08}; the main difference is that we use arbitrary metric spaces instead of CW-complexes. This alternative approach allows to study this invariant in the framework of coarse geometry. For a pair $(G, P)$ where $G$ is a finitely generated group with a word metric and $P$ is a subgroup, $\tilde e (G,P)$ is the cardinality of $\Ends(G, P)$. In \cref{section:filtered:ends:coends} we explain the equivalence with Bowditch notion of coends~\cite{Bow02}.

\begin{theorem}[\cref{thm:endsFunctor2}]\label{intro:cor:qiinvariance:relativeends}
Let $X$ and $Y$ be metric spaces, and $C\subseteq X$ and $D\subseteq Y$. If $f\colon X\to Y$ is a quasi-isometry such that $\Hdist(f(C),D)$ is finite, then $\Ends(f)\colon \Ends(X,C)\to\Ends(Y,D)$ is a bijection.
\end{theorem}

\begin{proof}[Proof of~\cref{intro2:thm:final:application}]
Let $q\colon G \to H$ be a quasi-isometry. Since  $P$ is a qi-character\-istic subgroup, there is a qi-collection $\mathcal{P}$ of $G$ that contains $P$. By \cref{intro2:cor:characteristic:groups}  there is a qi-characteristic  collection $\calQ$ of $H$   such that $q\colon (G, \mathcal{P}) \to (H,\mathcal{Q})$ is a quasi-isometry of pairs. After composing $q$ with a left translation of $H$ if necessary, there is a subgroup $Q$ in $\mathcal{Q}$ such that $\Hdist_H(q(P), Q)<\infty$.   Then \cref{intro:cor:qiinvariance:relativeends} implies that $\tilde e(G,P) = \tilde e (H,Q)$.
\end{proof}
 
\subsection*{Organization} The rest of the article is organized in four parts.
\cref{sec:qi-characteristic} discusses the definitions of quasi-isometry of pairs,  qi-characteristic collections of subspaces in the context of metric spaces, and a   characterization of qi-characteristic collections of subgroups, \cref{thm:qi-characteristic-groups2}. \cref{sub:Qi-char02} discusses some examples of qi-characteristic collections  arising from the theory of relatively hyperbolic groups, \cref{cor:BDM-RelHyp} and \cref{cor:RemovingQIsubgroup}. That section also cites other examples and non-examples of qi-characteristic collections from the literature. \cref{sec:qi-characteristic-Thm} contains the proof of the main result of the article, \cref{thm:qi-characteristic}, from which \cref{intro2:cor:characteristic:groups} is an immediate corollary.  The last part of the article, \cref{section:ends:of:metric:pairs}, develops the notion of filtered ends, from a metric perspective, in order to prove \cref{thm:endsFunctor2}, which contains the statement of \cref{intro:cor:qiinvariance:relativeends}.

\subsection*{Acknowledgements} The authors thank the anonymous referee for comments and suggestions that improved the quality of the manuscript.  The authors thank Juan Felipe Rodriguez-Quinche  for comments and proof reading of parts of an earlier version of the  manuscript. We also thank Sam Hughes for comments on this work, and for bringing to our attention an improvement of Proposition~\ref{prop:finite:generation}. The authors thank Anthony Genevois for pointing out the work in~\cite{Ge19,GT21} that provides examples of qi-characteristic collections, see \cref{prop:example01}. We also thank  Nicholas Touikan for comments and helpful discussions on the topics of the note. E.M.P acknowledges funding by the Natural Sciences and Engineering Research Council of Canada, NSERC.   
LJSS was supported by grant PAPIIT-IA101221.

\section{Qi-characteristic collections} \label{sec:qi-characteristic}

This part discusses the definitions of quasi-isometry of pairs, and qi-characteristic collections of subspaces in the context of metric spaces. \cref{sub:Qi-char01} contains the main result of this part, \cref{thm:qi-characteristic-groups2},   which is a  characterization of qi-characteristic collections of subgroups. \cref{sub:Qi-char03} contains some results on quasi-isometries of pairs for future reference.

\begin{definition}\label{def:quasiIsometry}
Let $X$ and $Y$ be metric spaces. A map $q\colon X \to Y$ is an \emph{$(L, C)$-quasi-isometry} if:
\begin{enumerate}
    \item $q$ is $(L,C)$-coarse Lipschitz: 
    \[\dist(q(x),q(y)) \leq L \dist(x,y)+C\]
    \item There is an $(L,A)$-coarse Lipschitz map $\bar q \colon Y\to X$, called a quasi-inverse of $q$, such that:
    \[ \dist(\bar q \circ q (x), x) \leq C, \qquad \dist(q\circ \bar q (y), y) \leq C\]
    for all $x\in X$ and $y\in Y$.
\end{enumerate}
A map $q\colon X \to Y$ is an  \emph{$(L,C)$-quasi-isometric map} if the restriction $q\colon X \to q(X)$ is an $(L,C)$-quasi-isometry.
\end{definition}

The following definitions borrow ideas of the work of Kapovich and Leeb on the quasi-isometry invariance of the geometric decomposition of Haken manifolds~\cite[\S 5.1]{KaLe97}. These ideas have been used in a similar fashion in other works, for example~\cite{DS05, BDM09, MSW11, LFS15}.

\begin{definition}\label{defn:quasi-isometry-pairs}
Let $X$ and $Y$ be metric spaces, let $\mathcal{A}$ and $\mathcal{B}$ be   collections of subspaces of $X$ and $Y$ respectively.  A  quasi-isometry $q\colon X\to Y$ is a \emph{quasi-isometry  of pairs} $q\colon (X,\mathcal{A}) \to (Y,\mathcal{B})$ if there is $M>0$:
\begin{enumerate}
    \item For any $A\in \mathcal{A}$, 
    the set $\{ B\in \mathcal{B} \colon \Hdist_Y(q(A), B) <M \}$ is non-empty.      
    \item For any $B\in \mathcal{B}$, 
    the set $\{ A\in \mathcal{A} \colon \Hdist_Y(q(A), B) <M \}$ is non-empty. 
\end{enumerate}
In this case, if $q\colon X \to Y$ is a $(L,C)$-quasi-isometry, then $q\colon (X, \mathcal{A})\to (Y,\mathcal{B})$ is called a \emph{$(L,C,M)$-quasi-isometry}. If there is a quasi-isometry of pairs  $(X,\mathcal{A}) \to (Y,\mathcal{B})$ we say  that $(X,\mathcal{A})$ and  $(Y,\mathcal{B})$ are \emph{quasi-isometric pairs}.
\end{definition}

Quasi-isometries of pairs have recently attracted the attention of other researchers in group theory,  see for example~\cite{BuHr21, HaHr19, GT21, Ge19} (in the last reference the notion appears implicitly, see \cref{prop:example01}). 

\begin{definition}\label{def:qicharacteristic}
Let $X$ be a metric space with metric $\dist$. A collection of subspaces $\mathcal{A}$ is called \emph{quasi-isometrically characteristic}, or for short \emph{qi-characteristic}, if the following properties hold:
\begin{enumerate}
\item For any $L\geq 1$ and $C\geq 0$ there is $M=M(L,C)>0$ such that any $(L,C)$-quasi-isometry $q\colon X\to X$ is an $(L,C,M)$-quasi-isometry of pairs $q\colon (X, \mathcal{A}) \to (X, \mathcal{A})$.  

\item Every bounded subset $B\subset X$ intersects only finitely many non-coarsely equivalent elements of $\mathcal{A}$; where $A, A' \in \mathcal{A}$ are coarsely equivalent if their Hausdorff distance is finite.

\item  For any $A\in \mathcal{A}$ the set $\{A'\in \mathcal{A} \colon \Hdist(A, A')<\infty\}$ is bounded as a subspace of $(\mathcal{A}, \Hdist)$.
\end{enumerate}
\end{definition} 

\begin{remark}  
The first condition in \cref{def:qicharacteristic} could be interpreted as non-positively curved property of  $(X, \mathcal{A})$. That the constant $M$ only depends on $L$ and $C$ is reminiscent of the property that in  $\delta$-hyperbolic spaces the images of any pair of  $(L,C)$-quasi-geodesics $\mathbb{R} \to X$ are either at Hausdorff distance bounded by a constant $D=D(\delta, L, C)$, or they are at infinite Hausdorff distance.  On the other hand, there are simple examples of pairs $(X, \mathcal{A})$ where every $(L,C)$-quasi-isometry $q\colon X \to X$ is a $(L,C,M_q)$-quasi-isometry of pairs $(X, \mathcal{A})$, the set of constants $M_q$ is unbounded, and the second and third conditions of \cref{def:qicharacteristic} hold for  $\mathcal{A}$; see \cref{NonExamples}\eqref{Example01}. \end{remark}

The verification of the following proposition is left to the reader.

\begin{proposition}\label{prop:SecondMistake}
Suppose that $q\colon (X, \mathcal{A}) \to (Y, \mathcal{B})$ is a  quasi-isometry of pairs.
Then $\mathcal{A}$ is qi-characteristic if and only if $\mathcal{B}$ is qi-characteristic.
\end{proposition}

 \begin{example}
 \label{NonExamples}   The following examples illustrate that conditions (1)-(3) in \cref{def:qicharacteristic} are independent of each other.  
 \begin{enumerate} 
\item \label{Example01} Consider the pair $(X, \mathcal{A})$ where $X$ is the $n$-dimensional Euclidean space and $\mathcal{A}$ consists of a single compact subset of $X$. Note that  any quasi-isometry $q\colon X \to X$ is a quasi-isometry of the pair $q\colon (X, \mathcal{A})  \to (X, \mathcal{A})$. On the other hand, the collection $\mathcal{A}$ is not qi-characteristic in $X$, the first condition fails by considering translations,  while the  conditions (2) and (3) hold.

\item Let $X=\mathbb H^n$ be the n-dimensional hyperbolic space and let $\mathcal A$ be the set of all geodesic lines in $X$. Then condition (2) does not hold while conditions (1) and (3) hold. Condition (1) holds because $\mathbb{H}^n$ is a hyperbolic space in the sense of Gromov.

\item Let $X$ be the real line and let $\mathcal{A}=\{ \{n\} \colon n\in \Z \}$. Then conditions $(1)$ and $(2)$ hold, but $(3)$ does not.
\end{enumerate}
 \end{example}

\subsection{Qi-characteristic collections of subgroups} \label{sub:Qi-char01}

\begin{definition}\label{def:group:qicharacteristic} Let $G$ be a finitely generated group, and let $\mathcal{P}$ be a collection of subgroups of $G$. The collection $\mathcal{P}$ is  \emph{qi-characteristic} if $G/\mathcal{P}$ is a qi-characteristic collection of subspaces of $G$. 

A subgroup of $G$ is a \emph{qi-subcharacteristic subgroup} if it belongs to a qi-characteristic collection of subgroups of $G$.
\end{definition}

\begin{remark}[Simplified Notation]
Let $G$ be a finitely generated group, and let $\mathcal{P}$ be a collection of subgroups.  For the rest of this section,
by a quasi-isometry of pairs $(G, \mathcal{P}) \to (G, \mathcal{P})$ we mean a quasi-isometry of pairs $(G, G/\mathcal{P}) \to (G, G/\mathcal{P})$ in the sense of \cref{defn:quasi-isometry-pairs}. 
\end{remark}

Recall that the \emph{commensurator} of a subgroup $P$ of a group $G$ is the subgroup of $G$ defined as  
\[Comm_G(P)=\{g\in G \colon P\cap gPg^{-1} \text{ is a finite index subgroup of $P$ and $gPg^{-1}$ }\}.\]

\begin{theorem}\label{thm:qi-characteristic-groups2}
Let $G$ be a finitely generated group. A collection of subgroups $\mathcal{P}$ is qi-characteristic if and only if 
\begin{enumerate}
    \item For any $L\geq 1$ and $C\geq 0$ there is $M=M(L,C)>0$ such that any $(L,C)$-quasi-isometry $q\colon G\to G$ is an $(L,C,M)$-quasi-isometry of pairs $q\colon (G, \mathcal{P}) \to (G, \mathcal{P})$.

    \item $\mathcal{P}$ is finite.
    
    \item Every $P\in \mathcal{P}$ has finite index in its commensurator.

\end{enumerate}
\end{theorem}

The statement of the theorem is a consequence of the following lemmas. 
\begin{lemma}\cite[Lemma 2.2]{MSW11}
\label{lem:Perpendicular}
Let $G$ be a finitely generated group and let $B$ and $C$ subgroups. Then for any $k>0$ there is $M>0$ such that 
\[ B\cap \mathcal{N}_k(C)  \subseteq \mathcal{N}_M(B\cap C)\]
where  $\mathcal{N}_k(C)$ and  $\mathcal{N}_M(B\cap C)$ denote the closed neighborhoods of $C$ and $B\cap C$ in $G$ with respect to $\dist_G$ respectively. 
\end{lemma}

\begin{lemma}\label{lem:Commensuartor-HDistance}
Let $G$ be a finitely generated group. 
For any subgroup $P$ of $G$ and $g\in G$,   $\Hdist(P,gP)<\infty$ if and only if $g\in Comm_G(P)$. 
\end{lemma}
\begin{proof}
The following two equivalences are immediate:
\begin{enumerate}
    \item $\Hdist(P,gP)<\infty$ if and only if $\Hdist(P, gPg^{-1})<\infty$.
    \item $P$ is contained in finite neighborhood of $P\cap gPg^{-1}$ if and only if $g\in Comm_G(P)$
\end{enumerate}

Suppose that $\Hdist(P, gPg^{-1})<\infty$.
\cref{lem:Perpendicular} implies that  $P$ is contained in finite neighborhood of $P\cap gPg^{-1}$. Therefore $g\in Comm_G(P)$.  Conversely, suppose $g\in Comm_G(P)$. Then $P$ is contained in finite neighborhood of $P\cap gPg^{-1}$. Analogously, since $g^{-1} \in Comm_G(P)$, using a left translation, it follows that 
$gPg^{-1}$ is contained in finite neighborhood of $P\cap gPg^{-1}$. The two last statements imply that $\Hdist(P, gPg^{-1})$ is finite.
 \end{proof}

\begin{lemma}\label{lem:UltimoPorFavor}
\label{prop:qiChar-BasicProperties3} \label{prop:qiChar-BasicProperties}
Let $\mathcal{P}$ be a finite collection of subgroups of a finitely generated group $G$. The following statements are equivalent:
\begin{enumerate}
    \item For any $P\in \mathcal{P}$, $P$ has finite index in its commensurator. 
    
     \item For any $P\in \mathcal{P}$, the set $\{Q\in G/\mathcal{P} \colon \Hdist(P,Q)<\infty\}$ is finite.
    
    \item For any $P\in \mathcal{P}$, the set $\{Q\in G/\mathcal{P} \colon \Hdist(P,Q)<\infty\}$ is bounded as a subspace of $(G/\mathcal{P}, \Hdist)$.

\end{enumerate}
\end{lemma}
\begin{proof}
To show that the first statement implies the second, suppose every $P\in \mathcal{P}$  has finite index in its commensurator.
By \cref{lem:Commensuartor-HDistance}, for any $P\in \mathcal{P}$ the set $\{gP\in G/P \colon \Hdist(P,gP)<\infty\}$ is finite.
Since the collection $\mathcal{P}$ is finite, for any $P\in \mathcal{P}$, the set $\{Q\in G/\mathcal{P} \colon \Hdist(P,Q)<\infty\}$ is a finite union
 finite sets, hence it is finite.  It is immediate that the second statement implies the third one. That  the third statement implies the first one is a consequence of   \cref{lem:Commensuartor-HDistance}.
\end{proof}

\begin{proof}[Proof of \cref{thm:qi-characteristic-groups2}]
Observe that if $\mathcal{P}$ is a qi-characteristic collection of subgroups of a finitely generated group, then $\mathcal{P}$ is finite. Then the theorem is a direct consequence of \cref{lem:UltimoPorFavor}.
\end{proof}

\subsection{Other remarks on quasi-isometric pairs}\label{section:commensurability} \label{sub:Qi-char03}

In this part we record a couple of propositions on quasi-isometries of pairs over  commensurable groups, and an additional proposition at the end, for future reference.   The results of this subsection  are not used in the rest of the article.

\begin{proposition}\label{prop:FiniteIndexSubgroup}
Suppose that $H$ is a finite index subgroup of a finitely generated group $G$. Let $\mathcal{P}$ be a finite  collection of subgroups of $G$. Then there is a finite collection of subgroups $\mathcal{Q}$ of $H$ such that the inclusion $H\hookrightarrow G$ induces a quasi-isometry of pairs  $q\colon (H, \mathcal{Q})$ to $(G, \mathcal{P})$.

Here $\mathcal{Q}=\{Q_i \colon i\in I \}$ where $Q_i = g_iP_ig_i^{-1}\cap H$ and 
$\{g_iP_i \colon i\in I\}$  is a collection of  representatives of the orbits of the $H$-action on $G/\mathcal{P}$ by multiplication on the left.
\end{proposition}
\begin{proof}
The finite  collection $\mathcal{P}$ is a collection of orbit representatives of the $G$-action on $G/\mathcal{P}$ by left multiplication. Since $H$ is finite index  in $G$, the induced $H$-action on $G/\mathcal{P}$ has finitely many orbits. Let $g_1P_1, \dots , g_kP_k$ be representatives of the $H$-orbits of $G/\mathcal{P}$, let $Q_i = g_iP_ig_i^{-1}\cap H$, and define  
\[\mathcal{Q}=\{Q_i \colon 1\leq i \leq k \}.\]

Since $H$ is finite index in $G$,
for any $g\in G$ and $P\in \mathcal{P}$,  the index of $gPg^{-1}\cap H$ as a subgroup of $gPg^{-1}$ is finite, in particular, $\Hdist_G(g^{-1}Pg\cap H, g^{-1}Pg)$ is finite. Let 
\begin{equation}\nonumber
D_1=\max\left\{ \Hdist_G(g_iP_ig_i^{-1}\cap H, g_iP_ig^{-1}_i) \colon 1\leq i\leq k \right \}
\end{equation}
and
\begin{equation}\nonumber
 D_2=\max \{ \dist_G(e, g_i) \colon 1\leq i\leq k\}.
\end{equation}
 Observe that for any $h\in H$ and $1\leq i\leq k$, 
\begin{equation}\label{ineq:010}  
\begin{split}
\Hdist_G\left ( h Q_i, hg_iP_i\right) 
 & = \Hdist_G\left ( g_iP_ig^{-1}_i\cap H, g_iP_i\right) \\
 & \leq  
\Hdist_G\left ( g_iP_ig^{-1}_i\cap H, g_iP_ig^{-1}_i\right) + \Hdist_G\left ( g_iP_ig^{-1}_i, g_iP_i \right) \\
 & \leq D_1 + D_2.
\end{split}
\end{equation}

To conclude we argue that the inclusion $\imath\colon H\hookrightarrow G$ is indeed a quasi-isometry of pairs $(H, H/\mathcal{Q}) \to (G, G/\mathcal{P})$.  Let $D> D_1+D_2$. 
Suppose that $P\in G/\mathcal{P}$. Since $P=hg_iP_i$ for some $h\in H$ and $g_iP_i$, let $R=hQ_i$ and observe that   inequality~\eqref{ineq:010} implies that $\Hdist_G(R,P)< D$ and   $R\in H/\mathcal{Q}$.  Analogously, suppose that $Q \in H/\mathcal{Q}$. Since $Q=hQ_i$ for some $h\in H$ and $Q_i$, let $S=hg_iP_i$
and note that $\Hdist_G (Q, S)< D$ by inequality~\eqref{ineq:010}, and $S\in G/\calP$. 
\end{proof}

\begin{example}
 Let $G$ be a finitely generated infinite group with finitely many conjugacy classes of finite subgroups, and let $\mathcal{P}$ be a collection of representatives of conjugacy classes of finite subgroups. Suppose that $H$ is a torsion-free finite index subgroup of $G$. If $\mathcal{Q}$ is the set containing only the trivial subgroup of $H$, then inclusion $H\hookrightarrow G$ is  a quasi-isometry of pairs $(H, \mathcal{Q}) \hookrightarrow (G, \mathcal{P})$. 
\end{example}

\begin{proposition}\label{prop:FiniteIndexSupergroup}
Suppose that $H$ is a finite index subgroup of a finitely generated group $G$. Let $\mathcal{Q}$ be a finite collection of subgroups of $H$. Suppose that for each $Q_1\in \mathcal{Q}$ and   $g\in G$, there is  $Q_2\in \mathcal{Q}$ and $h\in H$ such that $gQ_1g^{-1} = hQ_2h^{-1}$. 
Then the inclusion $H\hookrightarrow G$ is a quasi-isometry of pairs $(H,\mathcal{Q}) \hookrightarrow (G, \mathcal{Q})$.
\end{proposition}
\begin{proof}
Consider the left $H$-action on $G/\mathcal{Q}$ by multiplication on the left. Note that this $H$-action preserves the Hausdorff distance $\Hdist_G$ between subsets of $G$. Since $\Hdist$ is a metric on $G/\mathcal{Q}$, the $H$-action on $G/\mathcal{Q}$ is by isometries. By definition, $H/\mathcal{Q}$ is a subset of $G/\mathcal{Q}$. 

To prove that the inclusion $H\hookrightarrow G$ is a quasi-isometry of pairs $(H, \mathcal{Q}) \to (G, \mathcal{Q})$, it is enough to verify that
\[ \max\left\{ \Hdist_G(gQ, H/\mathcal{Q}) \colon gQ\in G/\mathcal{Q} \right\} < \infty,\]
where
\[\Hdist_G(gQ, H/\mathcal{Q}):= \min\left\{\Hdist_G(gQ, hQ') \colon hQ'\in H/\mathcal{Q} \right\}.\]

Since $H$ is finite index subgroup of $G$, and $\mathcal{Q}$ is a finite collection, it follows that the $H$-action on $G/\mathcal{Q}$ has finitely many orbits. Let $\mathcal{R}$ be a collection of orbit representatives of the $H$-action on $G/\mathcal{Q}$.

Let $gQ \in \mathcal{R}$.
By hypothesis, there is $h\in H$ and $Q'\in \mathcal{Q}$ such that $gQg^{-1}=hQ'h^{-1}$. Therefore 
\[
\begin{split}
\Hdist_G(gQ,hQ')& \leq \Hdist_G(gQ, gQg^{-1}) + \Hdist_G(hQ'h^{-1}, hQ')\\ & \leq \dist_G(e,g)+\dist_G(e,h)<\infty
\end{split},\]
and hence 
\[
\Hdist_G(gQ,H/\mathcal{Q})<\infty.    
\] 
Since $\mathcal{R}$ is a finite set, 
\[ D = \max\{\Hdist_G(gQ,H/\mathcal{Q}) \colon gQ \in \mathcal{R}\} <\infty \]
is a well defined integer.

Since $H$ is a
subgroup of $G$, the subset $H/\mathcal{Q}$ of $G/\mathcal{Q}$ is $H$-invariant. Therefore
\[\Hdist_G(gQ, H/\mathcal{Q}) = \Hdist_G(hgQ, H/\mathcal{Q}) \]
for every $gQ\in \mathcal{R}$ and $h\in H$. 

Since $\mathcal{R}$ is a collection of representatives of orbits of $G/\mathcal{Q}$, it follows 
\[\Hdist_G(gQ, H/\mathcal{Q}) \leq D\]
for every $gQ \in G/
\mathcal{Q}$.
\end{proof}

\begin{proposition}\label{prop:finite:generation} 
Let $\calP$ and $\calQ$ be finite collections of  subgroups of the finitely generated groups $G$ and $H$ respectively. Suppose that $q\colon (G,\calP) \to (H, \calQ)$ is a quasi-isometry of pairs. Let $P \in \calP$, $Q\in \calQ$ and suppose $ \Hdist( q(P), hQ) <\infty$ for some $h\in H$.  If $P$ is finitely generated then $Q$ is finitely generated and quasi-isometric to $P$.
\end{proposition}
\begin{proof}
 By post-composing $q$ with the left multiplication by $h^{-1}$, we can assume that $\Hdist( q(P), Q) <\infty$.  Hence there is a quasi-isometry $g\colon (P, \dist_G) \to (Q, \dist_H)$. 
Let $S$ be a finite generating set of $P$, and let $\sigma=\max\{\dist_G(e, s)\colon s\in S\}$. Then it follows that $P$ is a $\sigma$-coarsely connected space with respect to $\dist_G$, that means for every pair of points $a,b$ in $P$ there is a sequence $a=x_0,x_1,\ldots ,x_\ell=b$ such that $\dist_G(x_i,x_{i+1})\leq \sigma$. Then the quasi-isometry $g$ implies that there is $\sigma'>0$ such that $Q$ is $\sigma'$-coarsely connected with respect to $\dist_H$. Let $\Delta$ be the graph with vertex set $Q$ and an edge between any pair of points at distance at most $\sigma'$ with respect to $\dist_H$. Then $\Delta$ is a connected graph, and since $(H,\dist_H)$ is  locally finite,  $\Delta$ is also a locally finite graph. Note that $Q$ acts freely and cocompactly on $\Delta$ and therefore, by the Schwartz-Milnor lemma, $Q$ is finitely generated.

Note that a finitely generated subgroup of a finitely generated group is coarsely embedded. Let $S$ and $T$ be finite generating sets of $P$ and $Q$ respectively. Then, we have that $\mathsf{Id}_P\colon (P,\dist_S)\to (P,\dist_G)$ is a coarse equivalence, and analogously for $\mathsf{Id}_Q$. Therefore $\mathsf{Id}_P\circ g\circ \mathsf{Id}_Q$ is a coarse equivalence between $P$ and $Q$. It is well known observation by Gromov that every coarse-equivalence between finitely generated groups is a quasi-isometry.
\end{proof}

A finitely generated subgroup $P$ of a finitely generated group $G$ is \emph{undistorted} if, by considering the corresponding words metrics induced by finite generating sets, the inclusion $P\hookrightarrow G$ is a quasi-isometric map.

 \begin{question}
 Let $P$ be a finitely generated qi-subcharacteristic subgroup of a finitely generated group $G$. Is $P$ an undistorted subgroup of $G$? 
 \end{question}

\section{QI-characteristic collections from Relative hyperbolicity, and other examples.} \label{sub:Qi-char02}

This part discuss examples of qi-characteristic collections arising from relative hyperbolicity, and cite other examples from the literature.  Following the convention in~\cite{BDM09}, if a group contains no collection of proper subgroups with respect to which is relatively hyperbolic, then we say that the group is \emph{not relatively hyperbolic (NRH)}. The following theorem is a   consequence of a corollary of work by 
Behrstock,   Dru\c{t}u,  and Mosher~\cite[Theorem~4.1]{BDM09} and  \cref{thm:qi-characteristic-groups2}.  

\begin{theorem}\label{cor:BDM-RelHyp}
Let $G$ be a finitely generated group hyperbolic relative to a finite collection $\mathcal{P}$ of NRH finitely generated subgroups. Then $\mathcal{P}$ is a qi-characteristic collection of $G$.
\end{theorem}
 
\begin{proof}
Relative hyperbolicity implies $\mathcal{P}$ is an almost malnormal collection \cite{Osin06}, in particular, $P=Comm_G(P)$ for any $P\in \mathcal{P}$. The collection $G/\mathcal{P}$ is coarsely discrete, that is, if $g_1P_1,g_2P_2\in G/\mathcal{P}$ are at finite Hausdorff distance then $g_1P_1=g_2P_2$. 

Observe that, if $g_1P_1, g_2P_2 \in G/\calP$, and $g_1P_1\subseteq \mathcal{N}_k(g_2P_2)$ for some $k$,
then $g_1P_1=g_2P_2$. Indeed, under the hypothesis, \cref{lem:Perpendicular} implies that $g_1P_1g_1^{-1}$ is a subset of a finite neighborhood of $g_1P_1g_1^{-1}\cap g_2P_2g_2^{-1}$; since $g_1P_1g_1^{-1}$ is an NRH, in particular it is infinite, and then malnormality implies that $g_1P_1g_1^{-1}=g_2P_2g_2^{-1}$; therefore $P_1=P_2$ and $g_2^{-1}g_1\in P_1$. 

The proof of~\cite[Theorem 4.8]{BDM09} shows that for every $L\geq 1$ and $C\geq 0$, there is $M=M(L,C,G, \mathcal{P})>0$ such that for any $A\in G/\mathcal{P}$ and any $(L,C)$-quasi-isometric embedding $q\colon A \to G$ there is $B\in G/\mathcal{P}$ such that $q(A) \subseteq \mathcal{N}_M(B)$. Hence, for any $(L,C)$-quasi-isometry $q\colon G \to G$ and any $A\in G/\mathcal{P}$ there is $B\in G/\mathcal{P}$ such that $q(A) \subset \mathcal{N}_M(B)$.

To conclude the proof we show that any $(L,C)$-quasi-isometry $q\colon G\to G$ is an $(L,C,LM+2C)$-quasi-isometry of pairs $q\colon (G,\calP) \to (G,\calP)$.  Let $A\in G/\calP$. By the statement of the previous paragraph, there is $B\in G/\calP$ such that \begin{equation}\label{eq:ultima:plz}
    q(A)\subseteq \mathcal{N}_M(B).
\end{equation} 
Let $\bar q\colon G\to G$ be a quasi-inverse of $q$ as defined in \cref{def:quasiIsometry}. It follows that $\bar q(B)\subseteq \mathcal{N}_M(A')$ where $A'\in G/\calP$, by the previous paragraph.  On the other hand, \eqref{eq:ultima:plz} implies that $A\subseteq \mathcal{N}_{LM+2C}(\bar q(B))$, and therefore $A\subseteq \mathcal{N}_{LM+M+2C}(A')$. Now, the statement of the second paragraph implies that $A=A'$. It follows that $\bar q(B)\subseteq \mathcal{N}_{M}(A)$, and by applying $q$ both sides we obtain that 
\[B\subseteq \mathcal{N}_{LM+2C}(q(A))\]
This last equation and \eqref{eq:ultima:plz} imply that $\Hdist(q(A),B)\leq LM+2C$.
\end{proof}

\begin{proposition}\label{prop:RemovingQIsubgroup}
Let $G$ be a finitely generated group and let $\mathcal{P}\cup\{H\}$ be a qi-characteristic collection of finitely generated undistorted subgroups. If for any $P\in \mathcal{P}$ there is no quasi-isometric embedding $P\to H$, then $\mathcal{P}$ is a qi-characteristic collection of $G$.
\end{proposition}
\begin{proof}
Let $L\geq 1$ and $C\geq 0$. Then there is $M=M(L,C)$ such that any $(L,C)$-quasi-isometry $q\colon G\to G$ is a $(L,C,M)$-quasi-isometry of pairs $q\colon (G, \mathcal{P}\cup\{H\}) \to (G, \mathcal{P}\cup\{H\})$. The assumption on $H$ implies $q(P)$ can not be at finite Hausdorff distance from a left coset of $H$. Therefore  $q\colon (G, \mathcal{P}) \to (G, \mathcal{P})$ is an $(L,C,M)$-quasi-isometry of pairs. 
\end{proof}

\begin{corollary}\label{cor:RemovingQIsubgroup}
Let $G$ and $\mathcal{P}=\{P_0,\ldots, P_n\}$ be as in \cref{cor:BDM-RelHyp}. Suppose that there is no quasi-isometric embedding $P_i \to P_0$ for $1\leq i\leq n$. Then $\{P_1, \dots , P_n\}$ is a qi-characteristic collection. 
\end{corollary}
\begin{proof}
Since each $P\in \mathcal{P}$ is a  finitely generated NRH group, \cref{cor:BDM-RelHyp} implies that $\mathcal{P}$ is a qi-characteristic collection of $G$. The subgroups $\mathcal{P}$ are undistorted in $G$, see~\cite{DS05}. By \cref{prop:RemovingQIsubgroup}, it follows that $\{P_1,\ldots , P_n\}$ is a qi-characteristic collection of subgroup of $G$.
\end{proof}

\begin{example}\label{prop:example01}
Recall that a subgroup is qi-subcharacteristic if it belongs to a qi-characteristic collection. 
\begin{enumerate}
\item A finite subgroup of an   finitely generated infinite group is not qi-subcharacteristic. 

\item The NRH hypothesis of \cref{cor:BDM-RelHyp} is necessary, for instance, if $F$ is a free group of finite rank then a maximal cyclic subgroup is not qi-subcharacteristic. There is a quasi-isometry of $F$ that maps an  infinite geodesic preserved by a non-trivial element of $F$ to a geodesic that is  preserved by no element of $F$.

\item Let $A$ and $B$ be finitely generated NRH groups endowed with word metrics with a common finite subgroup $C$. By \cref{cor:RemovingQIsubgroup}, if there is  no quasi-isometric embedding $A\to B$, then $\{A\}$ is a qi-characteristic collection of subgroups of $A\ast_C B$.

\item In contrast to the previous example, let $G=\mathbb Z^2\ast \mathbb Z^2$ and let $H$ be the left hand side factor.  While $H$ is a qi-subcharacteristic subgroup by \cref{cor:BDM-RelHyp}, the collection $\{H\}$ is not qi-characteristic. The second and third conditions of the \cref{thm:qi-characteristic-groups2} hold, but the first does not. Specifically, a  quasi-isometry that flips the two factors sends $H$ to a space that is at infinite Hausdorff distance of any of its left cosets.

\item \label{example:BS}
Let $n\geq 2$ and consider the Baumslag-Solitar group $BS(1,n)=\langle a,t| tat^{-1}=a^n\rangle$.
The distorted cyclic subgroup $\langle a \rangle$ is not qi-characteristic since it has infinite index in its commensurator. 

The subgroup $\langle t\rangle$ does not form a qi-characteristic collection. We sketch the argument  using a construction that appears in the work of Farb and Mosher~\cite{FaMo98} on quasi-isometric rigidity of solvable Baumslag-Solitar groups. They use a particular metric on the Cayley complex $X_n$ of $BS(1,n)$ together with the projection $\pi\colon X_n\to T_n$ to the Bass-Serre tree. Let us recall a few properties: the inverse image $\pi^{-1}(L)$ 
of any \emph{coherently oriented  proper line $L$ of $T_n$} is an isometrically embedded hyperbolic plane $H$; all hyperbolic planes of $X_n$ arise in this way and can be simultaneously identified with the upper half plane model of $\mathbb{H}^2$ so that inverse image $\pi^{-1}(x)$ for $x\in L$ correspond to an  horocycle based at $\infty\in \partial \mathbb{H}^2$. In this way, the parabolic isometry $q\colon \mathbb{H}^2 \to \mathbb{H}^2$ given by  $z\mapsto z+1$ preserves horocycles based at $\infty$, and hence it induces an isometry $q\colon X_n \to X_n$ such that $\pi\circ q =\pi$. The isometry $q$ preserves each hyperbolic plane of $X_n$, and each of these planes corresponds to a unique left coset of $\langle t \rangle$ which can be identified with a particular vertical geodesic. Since any two hyperbolic planes of $X_n$ are at infinite Hausdorff distance, and any two distinct geodesics of $\mathbb{H}^2$ are at infinite Hausdorff distance, it follows that $q(\langle t\rangle)$ is at infinite Hausdorff distance of every left coset of $\langle t \rangle$.
 
\item Consider an amalgamated product $G=\Z^3 \ast_\Z \Z^3$, where $\Z$ corresponds to a maximal infinite cyclic subgroup in both factors, and let $\mathcal{P}$ be the collection consisting of the two $\Z^3$ factors.  The work of Papasoglu~\cite[Theorem 7.1]{Pa05} implies that every $(L,C)$-quasi-isometry of $q\colon G\to G$ is $(L,C,M_q)$-quasi-isometry of pairs $q\colon (G, \mathcal{P}) \to (G, \mathcal{P})$ for some constant $M_q$. To show that $\mathcal{P}$ is qi-characteristic we need to show that $M_q$ can be chosen so that it depends only of $L$ and $C$, and not on $q$. We do not know whether the constant $M_q$ can be chosen so that it only depends on $L$ and $C$. Provided that is true, $\calP$ would be another example of a qi-characteristic collection.
\item \label{Ge01} If $F$ is a finite group and $H$ is a finitely presented one-ended group, then the collection consisting of only the subgroup $H$ of the wreath product $G=F\wr H$ forms a qi-characteristic collection. This is a result of Genevois and  Tessera \cite[Theorem 1.88 and Proof of Theorem 7.3]{GT21}.
\item \label{Ge02} Certain  graph products of finite groups contain qi-characteristic collections, see~\cite[Fact 3.14]{Ge19}.
\end{enumerate}
\end{example}

\section{Quasi-isometry Invariance of QI-Characteristic collections}\label{sec:qi-characteristic-Thm}

In this section we prove the following theorem which is the main result of  the article.

\begin{theorem}\label{thm:qi-characteristic}
Let $X$ be a metric space and let $\mathcal{A}$ be a qi-characteristic collection of subspaces.
If $G$ is a finitely generated group quasi-isometric to $X$, then there is a finite collection of subgroups $\mathcal{P}$ such that   $(G, G/\mathcal{P})$ and $(X, \mathcal{A})$ are quasi-isometric pairs.
\end{theorem}

The following corollary is an immediate consequence.

\begin{corollary}\label{cor:characteristic:groups}
Let $G$ be a finitely generated group, let $\mathcal{P}$ be a  qi-characteristic collection of subgroups of $G$.  If $H$ is a finitely generated group quasi-isometric to $G$, then there is a qi-characteristic collection of subgroups $\mathcal{Q}$ of $H$ such that $(H, \mathcal{Q})$ and $(G, \mathcal{P})$ have the same  quasi-isometry type. 
\end{corollary}

\begin{proof}[Proof of Theorem~\ref{thm:qi-characteristic}]
Let $q\colon G \to X $ and $\bar q\colon X \to G$ be $(L_0,C_0)$-quasi-isometries such that $q\circ \bar q$ and $\bar q\circ  q$ are at distance less than $C_0$ from the identity maps on $X$ and $G$ respectively. Let $x_0=q(e)$ where $e$ is the identity element in $G$ and assume, without loss of generality, that $\bar q(x_0)=e$. For $g\in G$, let $\mathbf{g}\colon G\to G$ denote the isometry  given by   $x\mapsto gx$; let $q_g\colon X\to X$ denote the composition $q_g= q \circ \mathbf{g}\circ\bar q$. Then the following statements can be easily verified:
\begin{itemize}
    \item For $g\in G$, $q_g\colon X\to X$ is an $(L,C)$-quasi-isometry where $L=L_0^2$ and $C=L_0C_0+C_0+1>C_0$. 
    \item ($G$ quasi-acts on $X$)  For $g_1,g_2\in G$, the map $q_{g_1g_2}$ is at distance at most $C$ from the map $q_{g_1}\circ q_{g_2}$; and the map $q_{g_1}\circ q_{g_1^{-1}}$ is at distance at most $2C$ from the identity.
    \item ($G$ acts $C_0$-transitively on $X$) For every $x ,y\in X$ there is $g\in G$ such that $\dist_G(x, q_g(y))\leq  C_0$.
\end{itemize}

For $A\in \mathcal{A}$, define  
\[  St (A) = \{ g\in G \colon \Hdist (q_g(A), A) < \infty \} .\] 

\begin{step}
For any $A\in \mathcal{A}$, $St(A)$ is a subgroup of G. 
\end{step} 
\begin{proof}
Let $g_1,g_2\in St(A)$, then  
\begin{equation}\nonumber
\begin{split}
      \Hdist(q_{g_1^{-1}g_2}(A), A) & \leq L \Hdist(q_{g_1} \circ q_{g_1^{-1}g_2}(A), q_{g_1}(A)) + C \\ 
      & \leq L \Hdist(q_{g_2}(A), q_{g_1}(A)) + LC + C \\
      & \leq L \Hdist(q_{g_1}(A), A) + L \Hdist(q_{g_2}(A), A)  + LC +C < \infty.
\end{split}
\end{equation} 
Hence $g_1^{-1}g_2 \in St(A)$. 
\end{proof} 

\begin{step}\label{step:translation} \label{step:Orbits} 
There is a constant $M_1>0$ and a finite subset $\mathcal{F}$ of $\mathcal{A}$ such that: 
\begin{enumerate}
    \item For any $A\in \mathcal{A}$ and $g\in G$ there there is $A'\in \mathcal{A}$ such that $\Hdist(q_g(A), A')<M_1$.
    \item For any $A\in \mathcal{A}$ and $a\in A$ there is  $g\in G$ and $B\in \mathcal{F}$ such that $\dist(q_g(a),x_0)\leq C_0$ and  $\Hdist(q_g(A), B)\leq M_1$ and  $\Hdist(q_{g^{-1}}(B) , A )\leq M_1$.
\item For any $A,A'\in \mathcal{A}$, $\Hdist(A, A')\leq M_1$ or $\Hdist(A, A') = \infty$.   
\end{enumerate}
 
\end{step}
\begin{proof}
Let $M=M(L,C)$ be provided by  Definition~\ref{def:qicharacteristic} for the collection $\mathcal{A}$. This constant satisfies the first item.  Since $\mathcal{A}$ is qi-characteristic, the collection 
\[  \mathcal{D} = \{A\in \mathcal{A} \colon \text{there is $a\in A$ such that $\dist(x_0, a) < M+C$} \} \]
contains only finitely many non-coarse equivalent elements. In particular, there is a finite subset $\mathcal{F}=\{B_1, \ldots , B_m\}$  of $\mathcal{A}$ such that any element of $\mathcal{D}$ 
is at finite Hausdorff distance from an element of $\mathcal{F}$. 
Since $\mathcal{A}$ is  qi-characteristic, there is a constant $K>0$ such that for any $A\in \mathcal{A}$ and $B\in \mathcal{F}$, if $\Hdist(A,B)<\infty$ then $\Hdist(A,B)<K$.

Let $A\in \mathcal{A}$ and $a\in A$.  Since $G$ quasi-acts $C_0$-transitively on $X$,  there is $g\in G$ such that $\dist(q_g(a), x_0) \leq C_0<C$.  Since $\mathcal{A}$ is qi-characteristic,  
there is $B'\in \mathcal{A}$ such that $\Hdist(q_g(A),B')< M$. 
Observe that $B' \in \mathcal{D}$, since   $\dist(q_g(a),x_0)\leq  C$ and $q_g(a)\in q_g(A)$ imply that
$\dist(x_0, B')\leq \dist(x_0, q_g(a))+\dist(q_g(a), B') \leq C+M$.  Hence  there is $B\in \mathcal{F}$ such that $\Hdist(B,B')<K$, and therefore
$\Hdist(q_g(A), B) \leq C+M+K$. Since  $\Hdist(A, q_{g^{-1}}(B)) \leq L \Hdist(q_g(A), B) + 2C$, the first and second statement hold with the constant $M_0=L(C+M+K) +2C$.  

For the third item, let $A, A' \in \mathcal{A}$ such that $\Hdist (A, A')<\infty$. Then, by previous paragraph, there is $g  \in G$ and $B\in \mathcal{F}$ such that $\Hdist(q_{g}(A), B)\leq M_0$. Since $\mathcal{A}$ is qi-characteristic,  
 there is $B'\in \mathcal{A}$ such that $\Hdist(q_g(A'), B')\leq M_0$. 
It follows that $\Hdist (B, B')$ is finite and hence bounded by $K$. Therefore
\begin{equation}\nonumber
\begin{split}
\Hdist(A, A') & \leq L\Hdist(q_g(A), q_g(A')) + C \\  & \leq L(2M_0+K)+C.
\end{split}
\end{equation} 
The proof concludes by defining $M_1$ as the constant on the right of the previous inequality. 
\end{proof}

\begin{step} \label{step:ConstantM2}
There is a constant $M_2$ with the following property. For every $A\in \mathcal{A}$
\[  St (A) = \{ g\in G \colon \Hdist (q_g(A), A) \leq M_2 \} .\]
\end{step}
\begin{proof}
Suppose $g\in St(A)$. Since $\mathcal{A}$ is qi-characteristic,  
there is $A'\in \mathcal{A}$ such that $\Hdist (q_g(A), A')\leq M_1$.  Since 
$\Hdist (A, q_g(A)) <  \infty$, we have that $\Hdist (A, A')<\infty$  and 
hence
\[ \Hdist (A, q_g(A))\leq \Hdist(A, A')+\Hdist(A', q_g(A)) \leq 2M_1   . \]
As a consequence, 
\[  St (A) = \left \{ g\in G \colon \Hdist (q_g(A), A) \leq 2M_1 \right \}\]
and to conclude let $M_2=2M_1$ is defined.
\end{proof}

From here on let $M$ be $\max\{M_1, M_2, LM_1+3C\}+1$.

\begin{step}\label{step:conjugation}
If $A,A'\in \mathcal{A}$, $g\in G$ and $\Hdist(q_g(A), A')$ is finite, then 
\begin{equation}\nonumber g St (A) g^{-1} =  St (A').\end{equation}
\end{step}
\begin{proof}
Let $h \in St(A')$. Note that
\begin{equation}\nonumber
\begin{split}
\Hdist&(q_{g^{-1}hg} (A), A) \leq \\
& \leq L\Hdist(q_h\circ q_g (A), q_g(A)) +5LC  \\ 
& \leq   L \Hdist(q_h\circ q_g (A), q_h (A')) + L\Hdist(q_h (A'), A')  + L\Hdist(A', q_g(A)) +5LC  \\ 
& \leq   L^2 \Hdist(  q_g (A), A' ) + L\Hdist(q_h (A'), A')  + L\Hdist(A', q_g(A)) + 6LC < \infty. 
\end{split}
\end{equation}
Hence $g^{-1}hg \in St (A)$, and we conclude that $g St (A) g^{-1} \supseteq  St (A')$.  The other inclusion is proved analogously.
\end{proof}

The following step is a version of~\cite[Lemma 5.2]{KaLe97}.

\begin{step}\label{step:transitive-stabilizers}
There is $D_0>0$ such that for any $A\in \mathcal{A}$ and $a\in A$, then  \[\Hdist(A, St (A)a)\leq D_0\] where $St (A)a=\{ q_g(a)| g\in St(A) \}$.
\end{step}
\begin{proof}
Recall $\mathcal{F}=\{B_1, \ldots , B_m\}$. Let
\[
I   =\{(i,j) \colon \text{there is $g\in G$ such that $\Hdist(q_g(B_i),B_j)\leq M$}\} 
\]
and note that
\[I= \{(i,j) \colon \text{there is $g\in G$ such that $\Hdist(q_g(B_i),B_j)< \infty$}\}
\]
by Step~\ref{step:translation}.
 For each $(i,j)\in I$ choose $g_{i,j} \in G$ such that $\Hdist(q_{g_{i,j}}(B_i),B_j)\leq M$.
Let \[T=\max\{ \dist(q_{g_{i,j}}(x_0), x_0) \colon (i,j) \in I\} <\infty.\] 

The constant $D_0=D_0(L,C, M, T)$ is defined at the end of the proof and it is larger than $M$. 
Let $A\in \mathcal{A}$ and $a\in A$. 
Observe  $St (A)a$ is contained in the $M$-neighborhood of $A$.  We show below that for any $b\in A$ there is $\gamma\in St(A)$ such that $\dist (q_\gamma (a), b)\leq D_0$, which implies $\Hdist(A, St (A)a)\leq D_0$.

Let $b\in A$. Since $G$ acts $C_0$-transitively on $X$, there are $\alpha, \beta \in G$ such that
\[ \dist (q_\alpha (a), x_0)\leq C_0 \quad \text{and } \dist(q_\beta (b), x_0) \leq C_0 .\]
By Step~\ref{step:translation}, there are $B_i$ and $B_j$ in $\mathcal{F}$ such that $\Hdist(q_\alpha(A),B_i)\leq M$ and $\Hdist(q_\beta(A), B_j)\leq M$, therefore $(i,j)\in I$. To simplify notation, let $g$ denote the corresponding element $g_{i,j}$. Let  $\gamma= \beta^{-1} g \alpha$ and note that $\gamma \in St (A)$ since 
\begin{equation}\nonumber
\begin{split}
 \Hdist(q_\gamma(A), A) 
 & \leq L\Hdist(q_g \circ q_\alpha (A), q_\beta (A) ) + 4LC \\ 
 & \leq  L\Hdist(q_g \circ q_\alpha (A), q_g  (B_i) ) + L\Hdist(q_g  (B_i),  B_j ) + L\Hdist(B_j, q_\beta (A) ) +4LC\\
 & \leq (L^2M+LC) + LM + LM +4LC.
\end{split}
\end{equation}
 Moreover,   
\begin{equation}\nonumber
\begin{split}
\dist(q_\gamma (a),b) 
&  \leq L\dist(q_g\circ q_\alpha(a), q_\beta ( b) ) + 4 LC \\
&  \leq L\dist(q_g\circ q_\alpha(a), q_g(x_0) ) + L\dist(q_g(x_0) , x_0 ) + L\dist(x_0,  q_\beta ( b) ) + 4LC \\
& \leq  (L^2C_0 +LC) +  LT +LC_0 + 4LC = D'. 
\end{split}  
\end{equation}
Let $D_0=\max\{D',M\}$.
 \end{proof}

\begin{step}\label{step:FiniteOrbits}
Let $St(\mathcal{A}) = \{ St_M(A) \colon A\in \mathcal{A} \}$.
Then $G$ acts (from the left) on $St(\mathcal{A})$ by conjugation with finitely many orbits.  Moreover $\{St(A)\colon A\in \mathcal{F}\}$ contains a representative of each $G$-orbit.
\end{step}
\begin{proof}
First, we  verify that the action is well defined.
Let $g\in G$ and $A\in \mathcal{A}$. Since $\mathcal{A}$ is qi-characteristic  
 there is $A'$ such that $\Hdist(q_g(A),A')<M$. Then, by Step~\ref{step:conjugation}, we have that $gSt (A)g^{-1} = St (A')$.  

To verify that the action has finitely many orbits, let $A\in \mathcal{A}$. Then by Step~\ref{step:Orbits}, there is $g\in G$ and $B\in \mathcal{F}$ such that $\Hdist(q_g(A), B)\leq M$. Hence, by Step~\ref{step:conjugation}, $gSt (A)g^{-1}=St (B)$. Therefore $\{St (B)\colon B\in \mathcal{F}\}$ contains a collection of representatives of the $G$-orbits of $St(\mathcal{A})$; since $\mathcal{F}$ is finite the claim follows. 
\end{proof}

\begin{step}\label{step:computation}
Let $g\in G$, $U\subset G$ and $a,b\in X$. The the following statements hold.
\begin{itemize}
    \item $\dist (q(g), q_g(x_0))\leq LC_0+C_0$,
    \item $\Hdist (q(U), Ux_0) \leq LC_0+C_0$, and 
    \item $\Hdist (Ua, Ub) \leq L\dist (a,b) +C$. 
\end{itemize}
Here $Ux_0$ denotes the set $\{q_g(x_0) \colon g\in U\}$, and $Ua$ and $Ub$ are defined analogously.
\end{step}
\begin{proof}
For the first inequality, recall that  $\bar q(x_0)=e$, and note that 
\[ \dist (q_g(x_0), q(g)) \leq L_0 \dist(\mathbf{g}\circ \bar q (x_0), \bar q\circ q (g)) +C_0 \leq LC_0+C_0.\]
The second statement follows from the first one. The third inequality is a consequence of $q_g$ being an $(L,C)$-quasi-isometry for every $g\in G$.
\end{proof}

Define
\[  K_1= 3 ( LM+LC+M+10C +D_0), \qquad K_2= 3( LK_1+10C )> K_1 \]
and 
\[ D = 5 K_2  \]
Suppose 
\[ \mathcal{F}=\{B_1, \ldots , B_m\}\]
and let $P_i=St(B_i)$ and define 
\[\mathcal{P}=\{P_1, \ldots , P_m\}.\]

\begin{step}\label{step:final}
For any $g\in G$ and $P\in \mathcal{P}$, the set $\{A \in \mathcal{A} \colon \Hdist(q(g P), A) < \infty\}$ is bounded in $(\mathcal{A}, \Hdist)$. Moreover, there is  $A\in \mathcal{A}$ such that $\Hdist(q(g P), A) \leq D$.    
\end{step}
\begin{proof}
The first statement is a direct consequence of the third item of the definition of qi-characteristic. It is left to prove the existence of $A\in \mathcal{A}$ such that $\Hdist(q(g P), A) \leq D$.
By definition of $\mathcal{P}$, there is  $B\in \mathcal{F}$ such that $P=St(B)$. Since $\mathcal{A}$ is qi-characteristic, there is $A\in \mathcal{A}$ such that 
\begin{equation}\label{ineq:00}
\Hdist(q_g (B), A)< M<K_2.
\end{equation}
By the triangle inequality,  
\begin{equation}\label{ineq:step8}
\begin{split}
 \Hdist (q(g P), A) & \leq 
\Hdist (
q(g P), q_g ( q(P) ) 
)
+
\Hdist (
q_g ( q(P) ), q_g (B) 
)
+
\Hdist(
q_g (B), A
) \\
&  
\leq 
\Hdist (
q(g P), q_g ( q(P) )
)
+
\Hdist (
q_g ( q(P) ), q_g (B) 
)
+K_2
\end{split}    
\end{equation}
To conclude the proof, we show below that  the two terms on the last line of inequality~\eqref{ineq:step8} are bounded by
$K_2$, and hence $ \Hdist (q(g P), A)<D$.  For the first term,  
\begin{equation}
    \Hdist (q(g P), q_g ( q(P) ) \leq L_0\Hdist(P, \bar q \circ q (P))+C_0 \leq L_0C_0+C_0<K_2. 
\end{equation}
since $q$ is an $(L_0,C_0)$-quasi-isometry, the action of $g$ on $G$ is an isometry, and $\bar q\circ q$ is
at distance less than $C_0$ from  the identity. For the second term, observe that
\begin{equation}\label{ineq:01}
\begin{split}
    \Hdist (q_g ( q(P) ), q_g (B)  ) & \leq L\Hdist (q(P), B) +C 
\end{split}
\end{equation}
since $q_g$ is an $(L,C)$-quasi-isometry. To argue that 
  $ \Hdist (q_g ( q(P) ), q_g (B)  )<K_2$ is enough to show that $\Hdist (q(P), B) <K_1$.
Since $B\in \mathcal{F}$, there is $b\in B$ such that $\dist(x_0, b) < M+C$. The triangle inequality implies
\begin{equation}\label{ineq:02}
     \Hdist (q(P), B)   \leq     \Hdist(q(P), Px_0) +   \Hdist(Px_0, Pb) +   \Hdist(Pb, B)  .
\end{equation}
By Step~\ref{step:computation},
\begin{equation}\label{ineq:03}
\Hdist(q(P), Px_0) \leq LC_0+C_0,
\end{equation}
and     
\begin{equation}\label{ineq:04}
\Hdist(Px_0, Pb) \leq L \dist(x_0, b)+C \leq L(M+C)+C.
\end{equation}
By Step~\ref{step:transitive-stabilizers}, 
\begin{equation}\label{ineq:05}
\Hdist(Pb, B) \leq D_0.
\end{equation}
Then inequalities~\eqref{ineq:03},~\eqref{ineq:04}, and~\eqref{ineq:05} imply that the expression on the right of~\eqref{ineq:02} is bounded by $K_1$ which completes the proof.
\end{proof}

 \begin{step}\label{step:final2}
For any $A\in \mathcal{A}$, the set $\{gP\in G/\mathcal{P} \colon \Hdist (q(g P), A) < \infty \}$ is  finite, and there is $gP\in G/\mathcal{P}$ such that $\Hdist (q(gP), A)\leq D$.  
\end{step}
\begin{proof}
Let  $A\in \mathcal{A}$. Let $P\in \mathcal{P}$ and suppose $\Hdist(q(P), A)<\infty$ and   $\Hdist(q(gP), A)<\infty$. Then
\[\Hdist(q_g(A), A) \leq \Hdist(q_g(A), q(gP)) + \Hdist(q(gP), A) < \infty \]
and hence $g \in St(A)$.  \cref{step:ConstantM2} implies that $\Hdist(gP, P) \leq LM+C$.  Since $G$ is locally finite, there are only finitely many left cosets $gP$ such that  $\Hdist(q(gP), A)<\infty$. Since $\mathcal{P}$ is finite, it follows that the set $\{gP\in G/\mathcal{P} \colon \Hdist(q(g P), A) < \infty \}$ is a finite union of finite sets and hence finite.
 
Now we   prove that there is $gP\in G/\mathcal{P}$ such that $\Hdist(q(gP), A)\leq D$.  By Step~\ref{step:Orbits} , there is  $g\in G$ and $B\in \mathcal{F}$ such that  $\Hdist(q_g(B), A)\leq M$. 
By Step~\ref{step:conjugation}, $St (A) =g St (B)g^{-1}$.
Let $P\in \mathcal{P}$ such that $P=St(B)$.  
Since $B\in \mathcal{F}$ there is $b\in B$ such that $\dist(x_0,b)\leq M+C$. Let $a\in A$ such that $\dist(q_g(b),a)\leq M$. It follows that 
\begin{equation}\nonumber
\begin{split}
\Hdist(gPb, gPg^{-1}a) &\leq \Hdist(gPb, gP q_{g^{-1}}(a)) +  \Hdist(gP q_{g^{-1}}(a), gPg^{-1}a) \\
& \leq  L\dist(b, q_{g^{-1}}(a))+C  +  C\\
& \leq  L(LM+C+2C)+2C \leq K_2.  
\end{split}
\end{equation}  
By Step~\ref{step:computation} and Step~\ref{step:transitive-stabilizers},
\begin{equation}\nonumber
\begin{split}
\Hdist&(q(gP), A) \leq \\
& \leq \Hdist(q(gP), gPx_0)+\Hdist(gPx_0, gPb)+\Hdist(gPb,gPg^{-1}a)+\Hdist( gPg^{-1}a, A)\\
& \leq (LC_0+C_0) + (L\dist(x_0,b)+C) + K_2 + D_0\\
& \leq K_2+K_2+K_2+D_0 < D. \end{split}  
\end{equation}
\end{proof}

To conclude the proof of the theorem, observe that $q\colon G\to X$ is a quasi-isometry of pairs $q\colon (G, G/\mathcal{P})\to (X, \mathcal{A})$ as a consequence of Steps~\ref{step:final} and~\ref{step:final2}.
\end{proof}

\section{Filtered Ends of Pairs}\label{section:ends:of:metric:pairs}

In this section, the following result is proved. 

\begin{definition}
Let $\mathsf{QPMet}$ be the category  whose objects are pairs $(X,C)$ where $X$ is a metric space and $C$ is a non-empty subspace;  morphisms $ (X,C)\to (Y,D)$ are  quasi-isometric maps $f\colon X\to Y$ such that $D \subseteq f(C)^{+r}$ for some $r\geq 0$, where $f(C)^{+r}$ is the $r$-neighborhood of $f(C)$; and the composition law is the standard composition of functions. 
\end{definition}

\begin{theorem}\label{thm:endsFunctor2}
There is a covariant functor
\[ \Ends\colon \mathsf{QPMet} \to \mathsf{Sets}\]
with the following properties:
\begin{enumerate}
    \item \label{thm:Ends:item:2} If  $f\colon (X,C) \to (X,C)$ is a morphism such that $f$ is at distance at most $r$ from the identity function on $X$, then $\Ends(f)$ is the identity morphism of $\Ends(X,C)$.
    \item \label{thm:Ends:item:3}  If $f\colon X\to Y$ is a quasi-isometry, $C\subseteq X$  and $D\subseteq Y$ and $\Hdist(f(C),D)$ is finite, then $\Ends(f)\colon \Ends(X,C)\to\Ends(Y,D)$ is a bijection.
    \item \label{thm:Ends:item:1} If $X$ is a proper geodesic  metric space  and $C\subseteq X$ is compact, then there is a natural bijection $\Ends(X,C) \to \mathsf{Ends}(X)$, where $\mathsf{Ends}(X)$ is defined as   in~\cite{BH99}. 
    \item \label{thm:Ends:item:4} If $X$ is the Cayley graph with the combinatorial path metric of a finitely genereted group $G$ and $P \leq G$, then the cardinality of $\Ends(X,P)$ coincides with the number of coends $\tilde e(G, P)$ as defined in Bowditch~\cite{Bow02}.
\end{enumerate}
\end{theorem}

\begin{definition}
Given a pair $(X,C)$ in $\mathsf{QPMet}$, the set $\Ends(X,C)$ is called the \emph{set of filtered ends of $(X,C)$}.
\end{definition}

\begin{remark}
In Geoghegan's book \cite[Section~14.3]{Ge08} there is an alternative approach to filtered ends that also coincides with Bowditch's approach~\cite{Bow02}. 
Geoghegan remarks the existence of a functor analogous to the one in \cref{thm:endsFunctor2}, specifically, from the category   whose objects are \emph{well-filtered CW-complexes of locally finite type} and morphisms are filtered maps. This alternative framework  is  suitable to study  filtered ends from an algebraic topological perspective.  Geoghegan's functor also satisfies properties (\ref{thm:Ends:item:1}) and (\ref{thm:Ends:item:4}) as a natural consequence of the definition, while properties (\ref{thm:Ends:item:2}) and (\ref{thm:Ends:item:3}) are not addressed as the book does not approach coarse geometry aspects of filtered ends. Our approach is suitable to study filtered ends from a coarse geometrical point of view.
\end{remark}

\subsection{Coarsely connected components}

Let $(X,\dist)$ be a metric space. A subset $A$ of $X$ is \emph{$\sigma$-coarsely connected} if for any pair of points $a,b\in A$ there is a finite sequence $a=x_0,x_1,\dots ,x_n=b$ of elements of $A$ such that $\dist(x_i,x_{i+1})\leq\sigma$ for each index $i$. We call such a sequence a \emph{$\sigma$-quasi-path} from $a$ to $b$ in $A$.

For $\sigma\geq 0$, let $\mathcal{C}(X,\sigma)$ denote the collection of $\sigma$-coarsely connected components of $X$.   Observe that if $\sigma'\geq \sigma$ then there is a  function
\begin{equation}\label{eq:IncreasingSigma} \rho\colon \mathcal{C}(X,\sigma) \to \mathcal{C}(X, \sigma')  \end{equation}
that satisfies  that $A\subseteq \rho(A)$ for any $A\in \mathcal{C}(X,\sigma)$. Let $\mathcal{C}_\infty (X,\sigma)$ be the collection of unbounded components in $\mathcal{C}(X,\sigma)$. Observe that if $\sigma=0$ then $\mathcal{C}_\infty(X,\sigma)$ is the empty set.

\begin{remark}\label{rem:rho-surjective}
For $\sigma\leq \sigma'$, the function \[\rho\colon \mathcal{C}_\infty(X,\sigma) \to \mathcal{C}_\infty(X, \sigma') \] is not necessarily  surjective nor injective. 
As an example, fix an integer $m>1$ and let $A=\{n\in \Z:n=0 \text{ or } n\geq m\}$ consider $X=\R\times A$ with the metric $d=\dist(\alpha\times \ell, \beta\times k)$ defined as follows: if $\alpha=\beta$ then $d=|\ell-k|$; if $\alpha \neq \beta$ then $d=|\alpha-\beta|+|\ell|+|k|$. In this case, $C_\infty(X, \sigma)$ has cardinality $1$ for $0<\sigma<1$, infinite cardinality  if $1\leq \sigma<m$, and cardinality $1$ again for $\sigma\geq m$. In particular $\rho$
is  not surjecive if $\sigma=1/2$ and $\sigma'=1$; and it is not injective if $\sigma=1$ and $\sigma'=m$. Note that the failure of surjectivity arises in the case that there are elements of $\mathcal{C}_\infty(X, \sigma')$ that are partitioned into bounded $\sigma$-connected components.
\end{remark}

\begin{proposition}\label{prop:induced:function:C}
If  $f\colon X\to Y$ is a $(\lambda,\epsilon)$-coarse Lipschitz map and $\sigma\geq 0$, then there is an induced function $f_*\colon \mathcal{C}(X,\sigma) \to \mathcal{C}(Y, \lambda\sigma+\epsilon)$ where $f_*(A)$ is the unique element of $\mathcal{C}(Y,\lambda\sigma+\epsilon)$ containing $f(A)$. In particular, there is an induced function\[ f_*\colon \mathcal{C}_\infty(X,\sigma)\to \mathcal{C}_\infty(Y,\lambda\sigma+\epsilon)\]  
\end{proposition}
\begin{proof}
Let $A\in \mathcal{C}(X,\sigma)$. Then $f(A)$ is  $\lambda\sigma+\epsilon$ coarsely connected and hence
there is a unique $B\in \mathcal{C}(Y,\lambda\sigma+\epsilon)$ such that $f(A)\subseteq B$. Observe that  if   $A\in \mathcal{C}(X,\sigma)$ is   unbounded, then $f(A)$ is unbounded and the second statement follows.  
\end{proof}

\subsection{Definition of the set of filtered ends $\Ends(X,C)$} 

Let $\langle I, \preceq \rangle$ denote the directed set where $I=(0,\infty)\times [0,\infty)$ and for $\alpha, \beta \in I$, 
\[ \alpha \preceq \beta \quad \text{if and only if } \alpha_1\leq \beta_1 \text{ and } \alpha_2 \geq \beta_2.\]
We use the following notation,  for $\alpha=(\sigma,\mu) \in I$ let 
\[ I_\alpha = \{(x,y) \in I \colon \sigma\leq x \text{ and } \mu \leq y \}, \]
and
\[J_\alpha = \{(\sigma,y) \in I \colon \mu \leq y \}. \]
For a positive real number $\sigma$, let
\[ I_{\sigma } = \{( x, y) \in I\colon  \sigma \leq x \text{ and } 0\leq y \} ,\]
and 
\[ J_{\sigma} = \{(\sigma, y) \in I\colon  0\leq y  \}. \]
\begin{remark}
Observe that $J_\alpha$ is a coinitial subset of $I_\alpha$. In particular,  $J_\sigma$ is a coinitial subset of $I_\sigma$. 
\end{remark}

Let $(X,\dist)$ be a metric space, and let $C$ be a subset of $X$. For $\alpha=(\sigma, \mu) \in  I$, let $C_\alpha$ denote the set of unbounded $\sigma$-coarsely connected components of $X- C^{+\mu}$, 
\[ C_\alpha =C_{(\sigma, \mu)} =   \mathcal{C}_\infty(X-C^{+\mu}, \sigma), \]
where $C^{+\mu}$ denote the open $\mu$-neighborhood of $C$.  

\begin{remark}\label{remark:uniqueness:function}
Let $A$ and $B$ sets, let $f\colon A \to B$ be a function. Suppose $\mathcal{A}$ and $\mathcal{B}$ are collections of subsets of $A$ and $B$ respectively. If any pair of distinct elements of  $\mathcal{B}$ are disjoint, then there is at most one function $g \colon \mathcal{A} \to \mathcal{B}$ with the property that $f(C)\subseteq g(C)$ for any $C\in \mathcal{A}$.
\end{remark}

\begin{proposition}\label{prop:rho-function}
Let $\alpha,\beta\in I$ and suppose that  $\alpha\preceq \beta$. Then there is a unique function 
\[ \rho_{\alpha,\beta} \colon C_\alpha \to C_\beta \]
that satisfies that $A\subseteq \rho_{\alpha,\beta}(A)$ for any $A\in C_\alpha$. In particular, if $\alpha\preceq \beta \preceq \gamma$ then \[\rho_{\beta, \gamma} \circ \rho_{\alpha,\beta} = \rho_{\alpha, \gamma}.\]
\end{proposition}
\begin{proof}
The uniqueness of the function is clear since $C_\beta$ is a collection of disjoint subsets of $X$, see Remark~\ref{remark:uniqueness:function}. To define the function we consider cases:
\begin{enumerate}
\item If $\alpha_2=\beta_2$, then $\rho_{\alpha,\beta}$ is a particular case of~\eqref{eq:IncreasingSigma}.  
\item If $\alpha_1=\beta_1$ then $\rho_{\alpha, \beta}$ is the function induced by the inclusion $X-C^{+\alpha_2} \to X-C^{+\beta_2}$. By Proposition~\ref{prop:induced:function:C},  we have that $A\subseteq \rho_{\alpha,\beta}(A)$ for any $A\in C_\alpha$. 
\item Suppose $\alpha_1<\beta_1$ and $\alpha_2>\beta_2$. Let $\gamma =(\alpha_1, \beta_2)$ and $\delta=(\beta_1, \alpha_2)$. 
Then $\alpha\preceq \gamma \preceq \beta$ and $\alpha\preceq \delta \preceq \beta$. Observe that $\rho_{\gamma, \beta} \circ \rho_{\alpha,\gamma}$ and $\rho_{\delta, \beta} \circ \rho_{\alpha,\delta}$ are functions from $C_\alpha$ to $C_\beta$ with the required property. By uniqueness, both functions are equal and they define $\rho_{\alpha,\beta}$.  \qedhere
\end{enumerate}
\end{proof}

\begin{definition}[Filtered Ends at scale $\sigma$]\label{def:FilteredEnds}
Let $X$ be a metric space,   $C$ be a subset of $X$, and $\sigma\geq 0$. The direct system $\langle C_i, \rho_{i,j} \colon  i,j \in J_{\sigma} \rangle$ is called the \emph{$\sigma$-ends-system for $(X,C,\sigma)$}, and the inverse limit  is denoted as  $(\Ends(X,C,\sigma), \psi_i)$; in symbols
\[ \Ends(X,C,\sigma) = \varprojlim_\mu \mathcal C_\infty(X-C^{+\mu}, \sigma) =\varprojlim_\mu C_{(\sigma,\mu)}=\varprojlim_{i\in J_\sigma} C_i.\]
 \end{definition}

For $\alpha\in I$, let
\[\Ends(X,C, J_\alpha) = \varprojlim_{i\in J_\alpha} C_i  \]
the inverse limit of the direct system $\langle C_i, \rho_{i,j} \colon  i,j \in J_\alpha \rangle$.

\begin{remark}\label{rem:EquivalentScales}
Let $\sigma>0$ and $r>0$. Then there is a canonical bijection
\[ \mathsf{Id}\colon \Ends(X,C,\sigma) \to \Ends(X,C^{+r},\sigma).\]
Indeed, for any $i,j \in J_\sigma$, if $i \preceq j$ then $J_i$ is a coinitial subset of $J_j$ and hence there is a natural bijection $\tau_{i,j}\colon \Ends(X,C,J_i) \to \Ends(X,C,J_j)$.  Moreover, if  $i,j,k\in J_\sigma$ and $i\preceq j \preceq k$, then $\tau_{j,k} \circ \tau_{i,j} = \tau_{i,k}$. Hence all the sets $\Ends(X,C,J_i)$ for $i\in J_\sigma$ can be naturally identified. 
\end{remark}

\begin{example} The following examples illustrate that for reasonably well-behaved metric spaces, the number of filtered ends stabilizes for large scales. See \cref{prop:injectivity:directed:system} for a general result for geodesic spaces.  
\begin{enumerate}
    \item Let $X$ be the Euclidean plane and let $C$ be an infinite line. Then $\Ends(X,C,\sigma)$ is a set of cardinality two if and only if $\sigma>0$. Note that $\Ends(X,C,0)$ is empty.
    \item Let $X$ be the set of pairs of integers with the Manhattan distance, that is, the distance between $(x_1,y_1)$ and $(x_2,y_2)$  is $|x_1-x_2|+|y_1-y_2|$. Let $C$ be the points of the form $(n,0)$ with $n\in \Z$. Then $\Ends(X,C,\sigma)$ has cardinality two  if $\sigma\geq 1$, and $\Ends(X,C,\sigma)$ is empty otherwise.
   \item Let $X$ be the subspace of the Euclidean plane consisting of two vertical lines at distance two, and two horizontal lines at distance one. 
   Let $C$ consists of a single point. Then $\Ends(X,C,\sigma)$ has cardinality $8$ if $\sigma<1$, $6$ if $1\leq \sigma <2$, and $4$ if $\sigma\geq 2$.     
\end{enumerate}
\end{example}

\begin{corollary} \label{cor:TransitionFunctions}
For any pair of non-negative real numbers $\sigma\leq \sigma'$, the inclusion of $I_{\sigma'}$ into $I_\sigma$ induces a canonical function \[\varphi_{\sigma,\sigma'} \colon  \Ends(X,C,\sigma) \to \Ends(X,C,\sigma').\]
In particular, if $\sigma\leq \sigma'\leq \sigma''$ then \[\varphi_{\sigma, \sigma''} = \varphi_{\sigma', \sigma''}\circ \varphi_{\sigma,\sigma'}.\]
\end{corollary}
\begin{proof}
Note that there is an isomorphism of directed systems $J_\sigma \to J_{\sigma'}$ given by $i=(\sigma,y)\mapsto j_*=(\sigma',y)$. 
In particular, if $i,j\in J_\sigma$ and $i \preceq j$, then there is a commutative diagram
\[
\begin{tikzcd}
 C_i \arrow[r, "\rho_{i,j}"] \arrow[d,"\rho_{i,i_*}"] & C_j \arrow[d,"\rho_{j,j_*}"]    \\
C_{i_*}   \arrow[r, "\rho_{i_*,j_*}"] &     C_{j_*}    \end{tikzcd}
\]
which, via the universal property of inverse limits, induces the morphism $\varphi_{\sigma,\sigma'}$.  
The first statement follows directly from~\cref{prop:rho-function}.
 The uniqueness in the universal properties implies that 
$\varphi_{\sigma, \sigma''} = \varphi_{\sigma', \sigma''}\circ \varphi_{\sigma,\sigma'}$ if $\sigma\leq\sigma'\leq\sigma''$. 
\end{proof}

\begin{definition}
The \emph{space of filtered ends $\Ends(X,C)$ of the pair $(X,C)$} is defined as the direct limit 
\[ \Ends(X,C) =  \varinjlim_{\sigma} \Ends(X,C,\sigma)  \]
 of  the system $\langle \Ends(X,C,\sigma), \varphi_{\sigma, \sigma'} \colon \sigma,\sigma'\in [0,\infty) \rangle$. This system is called the \emph{direct ends-system for $(X,C)$}. The \emph{number of filtered ends $\tilde e(X,C)$ of the pair $(X,C)$} is defined as the cardinality of $\Ends(X,C)$.    
\end{definition}

\subsection{Filtered ends in geodesic spaces}

In this part, we prove two results. First, that for geodesic metric spaces the number of filtered ends does not depend on the scale, \cref{prop:injectivity:directed:system}. We expect this result to hold for a broader class of metric spaces but a more general statement would not be discussed in this note.  

\begin{proposition}\label{prop:injectivity:directed:system}
Let $X$ be a geodesic metric space, and let $C$ be a subset of $X$.  For any $0< \sigma\leq \sigma'$, the function $\varphi_{\sigma,\sigma'}\colon  \Ends(X,C,\sigma) \to \Ends(X,C,\sigma')$ is a bijection. In particular, the induced function $\Ends(X,C,\sigma) \to \Ends(X,C)$ is a bijection.
\end{proposition}

The second result is that, for geodesic proper metrics spaces, if $C$ is compact then  $\Ends(X,C)$ is the standard set of ends of $X$. An analogous result in the development of filtered by Geoghegan can be found in \cite[Propositions~13.4.7~and~14.3.1]{Ge08}.

\begin{proposition}\label{prop:EndsEquivalence}
Let $X$ be a proper geodesic  metric space  and let $C$ be a non-empty compact subset of $X$, then there is a natural bijection $\Ends(X,C) \to \mathsf{Ends}(X)$ as defined in~\cite{BH99}. 
\end{proposition}

The proofs of    Propositions~\ref{prop:injectivity:directed:system} and~\ref{prop:EndsEquivalence} rely on the following definition and lemma. 

\begin{definition}[The set of $C$-proper infinite $\sigma$-rays $\mathcal P(X,C,\sigma)$]
An \emph{infinite $\sigma$-path $p$ in $X$} is an infinite sequence $x_0,x_1,\dots$ of points in $X$ such that $\dist(x_i,x_{i+1})\leq \sigma$ for all $i$. For such a path $p$, if $x_0\in C$ we say that $p$ starts at $C$; and we say that $p$ is \emph{$C$-proper $\sigma$-ray} if $\{n\colon x_n\in C^{+\mu} \}$ is finite for every $\mu\geq 0$. The set of all $C$-proper infinite rays $\sigma$-paths starting at $C$ is denoted as $\mathcal P(X,C,\sigma)$.
\end{definition}

\begin{lemma}\label{lemma:infinite:sigma:paths}
Let $\sigma\geq 0$. Then there is a surjective function $\psi_\sigma\colon\mathcal P(X,C,\sigma)\to \Ends(X,C,\sigma)$ with the following properties:
\begin{enumerate}
    \item If $\sigma'\geq \sigma$, then the following diagram commutes
    \[
\begin{tikzcd}
 \mathcal P(X,C,\sigma) \arrow[r,hook] \arrow[d,"\psi_\sigma"] & \mathcal P(X,C,\sigma') \arrow[d,"\psi_{\sigma'}"]    \\
\Ends(X,C,\sigma)   \arrow[r,"\varphi_{\sigma,\sigma'}"] &     \Ends(X,C,\sigma')    \end{tikzcd}
\]

\item Suppose $p,q\in \mathcal P(X,C,\sigma)$ and for every $\mu\geq 0$, there are points $x$ of $p$ and $y$ of $q$ that are connected by a $\sigma$-path in $X-C^{+\mu}$. Then  $\psi_\sigma(p)=\psi_\sigma(q)$.
\end{enumerate}
\end{lemma}
\begin{proof}
First we define the function  $\psi_\sigma\colon\mathcal P(X,C,\sigma)\to \Ends(X,C,\sigma)$.
Let $p$ be an element of $\mathcal P(X,C,\sigma)$, notice that for any $\mu\geq 0$, all but finitely many elements of $p$ belong to a unique element $A_{\mu}$ of $\mathcal{C}_\infty(X-C^{+\mu},\sigma)$. Moreover, if $\mu' \geq \mu$, then $A_{\mu' }\subseteq A_\mu$. It follows that the collection $(A_\mu)$, and in particular $p$, determines a unique element of $\Ends(X,C,\sigma)$ that we define as 
$\psi_\sigma(p)$. The two statements of the lemma now follow directly from the definition of $\psi_\sigma$ for $\sigma>0$.
\end{proof}

\begin{proof}[Proof of \cref{prop:injectivity:directed:system}]
If $X$ is bounded, then there is nothing to prove. Assume that $X$  is unbounded, and observe that $\Ends(X,C,\sigma)\neq\emptyset$ for every $\sigma>0$. 

Let us prove that $\varphi_{\sigma,\sigma'}$ is injective. Recall   $\Ends(X,C,\sigma)= \varprojlim_{i\in J_\sigma} C_i$ where  $C_i=\mathcal{C}_\infty (X-C^{+\mu}, \sigma)$ if $i=(\sigma, \mu)$. Let $A=(A_i)$ and $B=(B_i)$ be elements of  $\Ends(X,C,\sigma)$. Suppose that $\varphi_{\sigma,\sigma'}(A)=\varphi_{\sigma,\sigma'}(B)$. Let $k\in J_\sigma$, with $k=(\sigma, \mu)$. Let $\ell=(\sigma, \mu+\sigma')$ and observe that $A_\ell\subseteq A_k$ and analogously $B_\ell \subseteq B_k$. Now, let $m=(\sigma', \mu+\sigma')$ and observe that $A_\ell$ and $B_\ell$ are subsets of $A_m=B_m$; this last equality follows from the assumption $\varphi_{\sigma,\sigma'}(A)=\varphi_{\sigma,\sigma'}(B)$. Let $x\in A_\ell$ and $y\in B_\ell$. Then, since $x,y\in A_m$, there is a $\sigma'$-path $x=x_0,x_1,\dots,x_n=y$ in $X-C^{+\mu+\sigma'}$. By choosing a geodesic between any two consecutive points in the $\sigma'$-path one sees that, there is a polygonal continuous path from $x$ to $y$ in $X-C^{+\mu}$. It follows that there is a $\sigma$-path from $x$ to $y$ in $X-C^{+\mu}$, and therefore $A_k$ and $B_k$ are the same $\sigma$-coarsely connected component of $X-C^{+\mu}$. Since $k\in J_\sigma$ was arbitrary, it follows that $A=B$.

Now lets prove that $\varphi_{\sigma,\sigma'}$ is surjective. By \cref{lemma:infinite:sigma:paths}, it is enough to show that for every $p\in \mathcal P(X,C,\sigma')$ there is a $q\in \mathcal  P(X,C,\sigma)$ such that $\psi_{\sigma'}(p)=\psi_{\sigma'}(q)$. Let $m$ be such that $\sigma'\leq m\sigma$. Let $ p$ be the sequence $x_0,x_1,\dots$ and suppose it is an element of $\mathcal P(X,C,\sigma')$.  Since $X$ is geodesic, there is a finite $\sigma$-path $q_i$ from $x_i$ to $x_{i+1}$ of length at most $m$ for all $i$.  Then the $\sigma$-path $q$ formed by $q_0,q_1\dots$ is an element of $\mathcal  P(X,C,\sigma)$. It is a consequence of the second part of \cref{lemma:infinite:sigma:paths} that $\psi_{\sigma'}(p)=\psi_{\sigma'}(q)$. 
\end{proof}

\begin{proof}[Proof of \cref{prop:EndsEquivalence}]
Let $x_0\in C$. Any $C$-proper coarse ray $(x_i)_{i\in\N}$ induces a proper ray $r\colon [0,\infty)\to X$ such that $r(0)=x_0$ and  $r(i)=x_i$ for each $i\in\Z_+$ by concatenating geodesic segments from $x_i$ to $x_{i+1}$. By \cite[I.8 Lemma 8.29]{BH99}, any two geodesic rays represent the same element of $\mathsf{Ends}(X)$ if for every $\mu>0$  they can be connected by a $\sigma$-path in  $X-C^{+\mu}$. Then~\cref{lemma:infinite:sigma:paths} implies that there is a natural bijection $\Ends(X,C, 1) \to \mathsf{Ends}(X)$. The result follows applying \cref{prop:injectivity:directed:system}.
\end{proof}

\subsection{Equivalence of filtered ends of pairs $\tilde e(G,P)$ with Bowditch's coends}\label{section:filtered:ends:coends}

\begin{definition}
Let $G$ be a finitely generated group and let $P$ be a subgroup. Let  $X$ be a Cayley graph of $G$ with respect to a finite generating set and with the edge path metric.
\begin{enumerate}
    \item \emph{The number of filtered ends} $\tilde e(G,P)$ is defined as the cardinality of $\Ends(X,P)$. Note that this definition is independent of $X$ in view of \cref{cor:qiinvariance:relativeends}.
    
    \item (Bowditch's coends) The number of coends $\bar e(G,P)$ is defined as follows. Let $\mathscr{S}_0(P)$ be the collection of non-empty connected $P$- invariant subgraphs of $X$ with compact quotient under $P$. Given $A\in \mathscr{S}_0(P)$, let $\mathscr{C}_\infty(A)$ be the set of components of $X-A$ that are not contained in any other element of $\mathscr{S}_0(P)$. If $A\subseteq B$ and both are elements of $\mathscr{S}_0(P)$, then there is a surjective map $\mathscr{C}_\infty(B)\to \mathscr{C}_\infty(A)$. The cardinality of the inverse limit of the system $(\mathscr{C}_\infty(A)\colon A\in \mathscr{S}_0(P))$ is \emph{the number of coends of the pair $(G,P)$}.
\end{enumerate}
\end{definition}

\begin{remark}
Observe that  $P^{+\mu}$, the $\mu$-neighborhood of $P$, is an element of $\mathscr{S}_0(P)$. Moreover, $\mathscr{C}_\infty(P^{+\mu})$ coincides with $\mathcal{C}_\infty(X-P^{+\mu}, \sigma)$ for any $\sigma>0$ (see \cref{prop:injectivity:directed:system}). Since $(\mathcal{C}_\infty(X-P^{+\mu}, \sigma)\colon \mu>0)$ is a cofinal system of $(\mathscr{C}_\infty(A)\colon A\in \mathscr{S}_0(P))$, we conclude that the number of filtered ends and the number of coends of $(G,P)$ coincide.
\end{remark}

\subsection{Induced functions}

Let $f\colon X  \to  Y $ be a $(\lambda, \epsilon)$-quasi-isometric map. Let $C\subset X$ and $D\subset Y$ and suppose $D \subseteq f(C)^{+r}$ for some $r\geq0$. 

 For $a\geq0$, let 
\begin{equation} a^f = \lambda a+\epsilon.\end{equation}
 
Let $\alpha=(0,\lambda(\epsilon+r))$ and $\beta=(\epsilon,0)$, consider the order preserving bijective function 
\[ I_\alpha  \to I_\beta,  \qquad  (\sigma,\mu) \mapsto \left(\lambda\sigma+\epsilon, \frac1\lambda\mu-\epsilon-r  \right).\]
For $i \in I_\alpha$, let $i^f$ denote its image by this function.

Let $C_i$ and $D_i$ denote the objects of the $I_\alpha$-ends-system and $I_\beta$-ends-system of  $(X,C)$ and $(Y,D)$ respectively. 

\begin{lemma}\label{lemma:step1}
For each $i=(\sigma,\mu)\in I_\alpha$, there is a unique function $f_i\colon C_i \to D_{i^f}$ that satisfies $f(A)\subseteq f_i(A)$ for any $A\in C_i$. 
\end{lemma}
\begin{proof}
Since $f$ is a $(\lambda,\epsilon)$-quasi-isometric map,  if $\dist(x,C)\geq \mu$ then $\dist(f(x),f(C))\geq \mu/\lambda-\epsilon\geq r$ since $\mu\geq \lambda(\epsilon+r)$. Since $D \subseteq f(C)^{+r}$, it follows that $\dist(f(x),D)\geq \mu/\lambda-\epsilon-r\geq 0$.  Hence the restriction  $f\colon (X-C^{+\mu}) \to (Y- f(C)^{+(\mu/\lambda -\epsilon-r)})$ is a well-defined $(\lambda,\epsilon)$-coarse Lipschitz map and  Proposition~\ref{prop:induced:function:C} can be invoked to obtain $f_{i }$. Since $D_{i^f}$ is a collection of disjoint subsets of $Y$, the function $f_{i}$  is unique by Remark~\ref{remark:uniqueness:function}. 
\end{proof}

\begin{lemma}\label{lemma:step2}
For any $i,j\in I_\alpha$, if $i\preceq j$ then there is a commutative diagram
\begin{equation}\label{eq:CommSquare012}
\begin{tikzcd}
 C_i \arrow[r, "f_i"  ] \arrow[d] & D_{i^f} \arrow[d]    \\
 C_{j}  \arrow[r, "f_j"] &     D_{j^f}    
\end{tikzcd}
\end{equation}
where the vertical arrows are transition functions of the $\alpha$-ends-system and $\alpha^f$-ends-system of $(X,C)$ and $(Y,D)$ respectively.
\end{lemma}
\begin{proof}
 This is indeed a commutative diagram since $D_{j^f}$ is a collection of pairwise disjoint sets, and therefore there is only function $\xi\colon C_i\to D_{j^f}$ with the property that $f(A)\subseteq \xi (A)$ for any $A\in C_i$, see Remark~\ref{remark:uniqueness:function}.
\end{proof}

  \begin{definition}[Induced fuctions]\label{lemma:step3} \label{lemma:step4}
  Let $f\colon X  \to  Y $ be a $(\lambda, \epsilon)$-quasi-isometric map. Let $C\subset X$ and $D\subset Y$ and suppose $D \subseteq f(C)^{+r}$ for some $r\geq0$.   Let $0<a\leq b$. 
  
      The function $f_{a,b}\colon \Ends(X,C,a) \to \Ends(Y,D,b)$ is defined as the composition 
\[ f_{a,b} = \varphi_{a^f,b} \circ f_{a, a^f} , \]
where $\varphi_{a^f,a}\colon \Ends(Y,D,a^f)\to \Ends(Y,D,b)$ is a transition function provided by Corollary~\ref{cor:TransitionFunctions}. The function $f_{a,a^f}$ is defined as follows. The commutative squares \eqref{eq:CommSquare012} for $i$ and $j$ in $J_{(a, \lambda (\epsilon+r))}$ induce a function
  \[f_{a,a^f}\colon \Ends(X,C,a) \to \Ends(Y,D,a^f)\]
   by   taking inverse limits of the $D_{i^f}$'s and then the inverse limits of the $C_i$'s. Indeed, the inverse limits induce a function
    \[ f_{a, a^f}\colon\Ends(X,C, J_{(a,\lambda(\epsilon+r))}) \to \Ends(Y,D, J_{(a^f,0)}), \]
    and by Remark~\ref{rem:EquivalentScales}, $\Ends(X,C,a) = \Ends(X,C, J_{(a,\lambda(\epsilon+r))})$ and $\Ends(Y,D, J_{(a^f,0)}) =\Ends(Y,D,a^f)$.

\end{definition}

\begin{proposition}\label{lemma:step5}
Let $f\colon X  \to  Y $ be a $(\lambda, \epsilon)$-quasi-isometric map. Let $C\subset X$ and $D\subset Y$ and suppose $D \subseteq f(C)^{+r}$ for some $r\geq0$.

If $a\leq c$, $a^f\leq b$, $c^f\leq d$ and $b\leq d$, then there is a commutative diagram
    \[
\begin{tikzcd}
     \Ends(X,C,a) \arrow[r, "f_{a, b }"] \arrow[d] & \Ends(Y,D,b) \arrow[d]    \\
\Ends(X,C,c)   \arrow[r, "f_{c, d }"] &     \Ends(Y,D,d)    \end{tikzcd}
\]
 where the vertical arrows are transition functions of the direct ends-system of $(X,C)$ and $(Y,D)$ respectively.    
\end{proposition}
\begin{proof}
It is enough to verify that for $a\leq b$, there is a commutative diagram
    \[
\begin{tikzcd}
     \Ends(X,C,a) \arrow[r, "f_{a, a^f }"] \arrow[d] & \Ends(Y,D,a^f) \arrow[d]    \\
\Ends(X,C,b)   \arrow[r, "f_{b, b^f}"] &     \Ends(Y,D,b^f)    \end{tikzcd}
\]
 where the vertical arrows are transition functions of the direct ends-system of $(X,C)$ and $(Y,D)$ respectively. This follows    \cref{lemma:step3} and the commutative squares~\eqref{eq:CommSquare012}. 
\end{proof}

\begin{proposition}\label{prop:IdentityMorphism}
For the identity map $\mathsf{Id} \colon (X,C) \to (X,C)$ and $0<a\leq b$, the function $\mathsf{Id}_{a,b}\colon \Ends(X,C,a) \to \Ends(X,C,b)$ is $\varphi_{a,b}$.
\end{proposition}
\begin{proof}
Note that $\mathsf{Id}_i\colon C_i \to C_i$ satisfies that $A=\mathsf{Id}_i(A)$ for every $A\in C_i$ and hence by \cref{prop:rho-function}, $\mathsf{Id}_i = \rho_{i,i}$ for every $i\in J_a$. Since $\varphi_{a,a}$ and $\mathsf{Id}_{a,a}$ are both  induced by taking the inverse  limits of $\rho_{i,i}\colon C_i \to C_j$, see \cref{lemma:step3} and \cref{cor:TransitionFunctions} respectively, it follows that $  \mathsf{Id}_{a,a}=\varphi_{a,a}$, and therefore $\mathsf{Id}_{a,b}=\varphi_{a,b}$.
\end{proof}

\begin{proposition}\label{prop:UltimaDeVerdad}
Let $f\colon X\to X$  be a function at distance at most  $r$ from the identity. If  $0<a$ and $a^f\leq b$, then $f_{a,b}\colon \Ends(X,C,a) \to \Ends(X,C,b)$ is $\varphi_{a,b}$.
\end{proposition}
\begin{proof}
It is enough to prove that $f_{a,a^f}\colon \Ends(X,C,a) \to \Ends(X,C,a^f)$ is $\varphi_{a,a^f}$.
Consider $f\colon X\to X$ as a $(1,2r)$-quasi-isometric map. from the identity and let $\sigma > 0$.  Consider the induced functions $f_i \colon C_i \to C_{i^f}$ for $i\in J_\sigma$; here $i\in J_{(0,3r)}$ and if $i=(\sigma,\mu)$ then $i^f=(\sigma+2r, \mu-3r)$. 
 
By definition, the functions  $f_{a, a^f}$ and $\varphi_{a, a^f}$  are obtained as the inverse limits of the functions $f_i\colon C_i \to C_{i^f}$ and $\rho_{i,i^f}$ respectively. To complete the proof we show that $f_i=\rho_{i, i^f}$.  By definition of $f_i$, $f(A) \subset f_i(A)$ for any $A\in C_i$.  
The assumption on $f$ implies that $\Hdist(A, f(A))\leq r$. Since $f_i(A)$ is an unbounded $(\sigma+2r)$-coarsely connected component of $X-C^{\mu-3r}$ and $\Hdist(A, f(A))\leq r$, it follows that $A\subset f_i(A)$. By Proposition 4.5, $f_i=\rho_{i,i^f}$.
\end{proof}

\begin{proposition}\label{lemma:step6}
Let $f\colon X  \to  Y $,  $g\colon Y \to Z$ and $h=g\circ f$ be a quasi-isometric maps with some chosen constants. Let $C\subset X$,  $D\subset Y$ and $E\subset Z$ and suppose $D \subseteq f(C)^{+r}$,  $E \subset g(D)^{+r'}$  and $E\subset h(C)^{+r''}$ for some $r,r',r''\geq0$. 
 
Let $a>0$. Let $b,c$ and $d$ denote $ a^f, a^{g\circ f}, $ and  $(a^f)^g$ respectively. Then, by choosing the quasi-isometric  constants for $h$ appropriately, one can assume that $c\leq d$, and there is a  commutative diagram,
\[
\begin{tikzcd}
 & \Ends(Y,D,b) \arrow[r, "g_{ b,d }"] &   \Ends(Z,E,d)    \\
\Ends(X,C,a) \arrow[ur, "f_{a,b}"] \arrow[rr, "(g\circ f)_{a,c}"] & &    \Ends(Z,E,c) \arrow[u]  \\
&  &     \end{tikzcd}
\]
where the vertical arrow into $\Ends(Z,E,d)$ is a transition function of the direct ends-system for $(Z,E)$. 
\end{proposition}
\begin{proof}
If $i,j\in I_\alpha$ and $i\preceq j$ then  $i^{g\circ f}\preceq (i^g)^f$ and  there is a commutative diagram
\begin{equation}\label{eq:Cube}
\begin{tikzcd} 
    C_i \ar[rr, "f_i"]\ar[dr, "(g\circ f)_i"] \ar[dd] & & D_{i^f}\ar[dr,"(g\circ f)_{i^f}"]\ar[dd] & \\
    & E_{i^{g\circ f}} \ar[rr,crossing over]  & & E_{(i^f)^g}\ar[dd] & \\
    C_{j}\ar[rr,"f_j" near end] \ar[dr, "(g\circ f)_j"']& & D_{j^f}\ar[dr,"(g\circ f)_{j^f}"  near start] & \\
    & E_{j^{g\circ f}}\ar[rr] \arrow[from=uu, crossing over]  & & E_{(j^f)^g} &
\end{tikzcd}
\end{equation}
where all vertical arrows and all arrows in the front face are transition functions, and all other arrows are induced by $f$, $g$ and $g\circ f$ accordingly. Indeed, observe that the back face and side faces are particular cases of \cref{lemma:step2}, the front face is a square diagram of transition functions in the ends-system for $(Z,E)$, so all these faces are commutative diagrams. That the top and bottom faces are commutative diagrams follows from the same argument as in \cref{lemma:step2}.  The proposition follows from the commutative cubes~\eqref{eq:Cube}   taking the corresponding inverse limits.
\end{proof}

\subsection{Filtered ends as a functor}

Consider the category $\mathsf{QPMet}$ whose objects are pairs $(X,C)$ where $X$ is a metric space and $C$ is a subspace;  morphisms $ (X,C)\to (Y,D)$ are  quasi-isometric maps $f\colon X\to Y$ such that $D \subseteq f(C)^{+r}$ for some $r\geq 0$; and the composition law is the standard composition of functions.

\begin{proposition}\label{thm:endsFunctor}
There is a covariant functor
\[ \Ends\colon \mathsf{QPMet} \to \mathsf{Sets}\]
that maps a pair $(X,C)$ to $\Ends(X,C)$.

\end{proposition}
\begin{proof}
Let $f\colon (X,C)\to (Y,D)$ be a morphism of $\mathsf{QPMet}$. \cref{lemma:step5} allows us to define the morphism $\Ends(f)\colon \Ends(X,C)\to \Ends(Y,D)$ as follows. For any $a>0$ and $b>a^f$ there is a function $f_{a,b}\colon \Ends(X,C,a)\to \Ends(Y,D,b)$, and then taking direct limits we obtain $\Ends(f)$.
That $\Ends(\mathsf{Id}_{(X,C)})=\mathsf{Id}_{\Ends(X,C)}$ is a consequence of  \cref{prop:IdentityMorphism}. For morphisms $f\colon (X,C)\to (Y,D)$ and $g\colon (Y,D)\to (Z, E)$, that $\Ends(g\circ f)=\Ends(g)\circ \Ends(f)$ is a consequence of \cref{lemma:step6}. 
\end{proof}

The following proposition follows directly from \cref{prop:UltimaDeVerdad}.

\begin{proposition}\label{prop:UltimaDeVerdad2}
Let $f\colon X\to X$  be a function at  finite distance from the identity. Then $\Ends(f)$ is is the identity map of $\Ends(X,C)$.
\end{proposition}

\subsection{Quasi-isometry invariance of Filtered Ends}

\begin{proposition}\label{cor:qiinvariance:relativeends}
Let $X$ and $Y$ be  metric spaces, and $C\subseteq X$ and $D\subseteq Y$. If $f\colon X\to Y$ is a quasi-isometry such that $\Hdist(f(C),D)$ is finite, then $\Ends(f)\colon \Ends(X,C)\to\Ends(Y,D)$ is a bijection.
\end{proposition}

\begin{proof}
Consider $g\colon Y\to X$ a quasi-isometry such that $g\circ f$ and $f\circ g$ are at finite distance from the corresponding identity functions.  Note that  $g\circ f$ and $f\circ g$  induced morphisms $(X, C) \to (X,C)$ and $(Y,D) \to (Y,D)$ of $\mathsf{QPMet}$ respectively. By \cref{thm:endsFunctor} both $\Ends(g)\circ \Ends(f)$ and $\Ends(f)\circ \Ends( g)$ are the identity morphisms of $\Ends(X,C)$ and $\Ends(Y,D)$ respectively and in particular they are bijections. Therefore $\Ends(f)$ is a bijection.
\end{proof}

\bibliographystyle{alpha} 
\bibliography{myblib}
\end{document}